\documentclass[pdflatex]{sn-jnl}
\usepackage{graphicx}%
\usepackage{multirow}%
\usepackage{amsmath,amssymb,amsfonts}%
\usepackage{amsthm}%
\usepackage{mathrsfs}%
\usepackage[title]{appendix}%
\usepackage{xcolor}%
\usepackage{textcomp}%
\usepackage{manyfoot}%
\usepackage{booktabs}%
\usepackage{algorithm}%
\usepackage{algorithmicx}%
\usepackage{algpseudocode}%
\usepackage{listings}%
\usepackage{array}
\usepackage{hhline}
\usepackage[shortlabels]{enumitem}
\theoremstyle{thmstyleone}%
\newtheorem{theorem}{Theorem}
\newtheorem{proposition}[theorem]{Proposition}%
\newtheorem{lemma}[theorem]{Lemma}%
\newtheorem{corollary}[theorem]{Corollary}%
\theoremstyle{thmstyletwo}%
\newtheorem{remark}{Remark}%

\theoremstyle{thmstylethree}%

\begin{document}


\title[Article Title]{Rational points and inflexions on Hermitian-relative curves}

\author[1]{\fnm{Masaaki} \sur{Homma}}\email{homma@kanagawa-u.ac.jp}

\author*[2]{\fnm{Seon Jeong} \sur{Kim}}\email{skim@gnu.ac.kr}
\equalcont{These authors contributed equally to this work.}


\affil[1]{\orgdiv{Department of Mathematics}, \orgname{Kanagawa University},   \city{Yokohama}, \postcode{221-8686},   \country{Japan}}

\affil*[2]{\orgdiv{Department of Mathematics and RINS}, \orgname{Gyeongsang National University},   \city{Jinju}, \postcode{52828},   \country{Korea}}

\abstract{In \cite{hom-kim2024}, we introduced the notion of a Hermitian-relative curve, which is a plane curve defined by $(x^{\sqrt{q}}, y^{\sqrt{q}}, z^{\sqrt{q}})A (x,y,z)^t =0$ with $A \in GL(3, \mathbb{F}_q)$. After investigating their basic properties, we classified those curves with two or more rational inflexions. In this paper, we first complete the classification of curves whose rational points are all inflexions by demonstrating the existence of curves with exactly one rational point. As a continuation, we then classify Hermitian-relative curves that possess at least one rational point which is not an inflexion. We categorize these curves according to the ordered pair consisting of the number of rational points and the number of inflexions. Finally, we enumerate all possible pairs and prove that each pair is realized by a suitable curve.}




\maketitle

\section{Introduction}
Let $q$ be an even power of a prime $p$.
Hermitian curves are significant objects in various fields of study. The most commonly utilized forms of their equations are $x^{\sqrt{q}+1} + y^{\sqrt{q}+1} + z^{\sqrt{q}+1} = 0$ or $x^{\sqrt{q}+1} = y^{\sqrt{q}}z + yz^{\sqrt{q}}$. In a more general representation, a Hermitian curve is defined by the following quadratic form:$$(x^{\sqrt{q}}, y^{\sqrt{q}}, z^{\sqrt{q}}) A \begin{pmatrix} x \\ y \\ z \end{pmatrix} = 0,$$subject to the condition $A = A^*$, where $A^*$ is the conjugate transpose defined by $(a_{ij})^* = (a_{ji}^{\sqrt{q}})^t$. For a more comprehensive treatment of Hermitian varieties in finite projective spaces, one may refer to Bose and Chakravarti \cite{bos-cha1966}.
In  \cite{hom-kim2024}, the authors generalized the aforementioned equation by replacing the matrix $A$ with a general nonsingular matrix in $GL(3,\mathbb{F}_q)$. This generalized object was referred to as a ``relative of a Hermitian curve", or more simply, a ``Hermitian-relative curve". In that paper, the authors presented the fundamental properties of those curves.
According to Pardini~\cite{par1986}, all Hermitian-relative curves, including the Hermitian curve itself, are projectively equivalent over the algebraic closure $\overline{\mathbb{F}}_q$ of $\mathbb{F}_q$, though this is certainly not the case over $\mathbb{F}_q$.
In~\cite{hom-kim2024}, we showed that every such curve has at least one 
$\mathbb{F}_q$-rational point; moreover, the number of these points is 
congruent to $1 \pmod{\sqrt{q}}$. In particular, we classified those 
curves possessing two or more rational inflexions.
It is a well-known fact that a Hermitian curve has $q\sqrt{q} + 1$ many $\mathbb{F}_q$-points. The tangent line at each $\mathbb{F}_q$-point has contact of order $\sqrt{q}+1$ with the curve, whereas the tangent line at any other point has contact of order $\sqrt{q}$.
Since any Hermitian-relative curve is projectively equivalent to a Hermitian curve over $\overline{\mathbb{F}}_q$, it follows that the tangent line at any point on a Hermitian-relative curve makes contact with the curve with multiplicity either $\sqrt{q}$ or $\sqrt{q}+1$.

A point of contact where the tangent line has multiplicity $\sqrt{q}+1$ is called an inflexion. Accordingly, every $\mathbb{F}_q$-point of a Hermitian curve is an inflexion. In contrast, a Hermitian-relative curve that is not a Hermitian curve possesses a smaller number of $\mathbb{F}_q$-points, some of which may not be inflexions. Since any Hermitian-relative curve is projectively equivalent to a Hermitian curve over the algebraic closure $\overline{\mathbb{F}}_q$, the total number of inflexion points over $\overline{\mathbb{F}}_q$ remains the same as that of a Hermitian curve. Therefore, it is mathematically meaningful to determine the number of inflexions specifically among the $\mathbb{F}_q$-rational points. For a Hermitian-relative curve $C$, we denote by $N_q(C)$ and $I_q(C)$ the numbers of $\mathbb{F}_q$-rational points and $\mathbb{F}_q$-rational inflexions on $C$, respectively. Clearly, it holds that $N_q(C) \ge I_q(C)$ by definition. We say that a curve $C$ is \emph{of type} $(n, i)$ if $N_q(C) = n$ and $I_q(C) = i$. For instance, a Hermitian curve $H$ satisfies $N_q(H) = I_q(H) = q\sqrt{q} + 1$; hence, $H$ is of type $(q\sqrt{q} + 1, q\sqrt{q} + 1)$.

Hirschfeld \cite{hir1979, hir1998} provided a classification of cubic curves; specifically, for the case $q=4$, all nonsingular curves whose equations do not contain an $xyz$ term are categorized as Hermitian-relative curves. These results are summarized in Table 2 of \cite{hom-kim2024}. According to this classification, the types of Hermitian-relative curves over $\mathbb{F}_4$ are given by
\begin{equation} \label{eq:types}
(9,9), (3,3), (1,1), (7,1), (5,1), (9,0), (3,0).
\end{equation}
The purpose of the present paper is to extend these results for $q=4$ to the case where $q$ is an arbitrary even power of a prime.

The remainder of this paper is organized as follows. Section 2 provides a brief review of the results from \cite{hom-kim2024}. In Section 3, we classify all Hermitian-relative curves with $N_q(C) = I_q(C)$. 
In Section 4, we establish a complete classification of Hermitian-relative curves based on their types.
\section{Preliminaries}
For each $A \in GL(3, \mathbb{F}_q)$, let $C_A$ be the curve defined by the equation
$$\begin{pmatrix} x^{\sqrt{q}} & y^{\sqrt{q}} & z^{\sqrt{q}} \end{pmatrix} A \begin{pmatrix} x \\ y \\ z \end{pmatrix} = 0.$$
 Note that $C_A = C_{\alpha A}$ for any $\alpha \in \mathbb{F}_q^*$; thus, a Hermitian-relative curve is uniquely determined by an element of $PGL(3, \mathbb{F}_q)$. 
 Two matrices $A$ and $B$ are said to be congruent (denoted by $A \sim B$) if there exists a nonsingular matrix $T \in GL(3, \mathbb{F}_q)$ such that $B = T^* A T$. 
 Under the coordinate transformation
 $$\begin{pmatrix} x' \\ y' \\ z' \end{pmatrix} = T^{-1} \begin{pmatrix} x \\ y \\ z \end{pmatrix},$$ 
 we obtain the following identity:
 $$\begin{pmatrix} x^{\sqrt{q}} & y^{\sqrt{q}} & z^{\sqrt{q}} \end{pmatrix} A \begin{pmatrix} x \\ y \\ z \end{pmatrix} = \begin{pmatrix} {x'}^{\sqrt{q}} & {y'}^{\sqrt{q}} & {z'}^{\sqrt{q}} \end{pmatrix} T^* A T \begin{pmatrix} x' \\ y' \\ z' \end{pmatrix}.$$ 
 This identity implies that the curves $C_A$ and $C_B$ are projectively equivalent. Consequently, we have the following:
\begin{lemma}[\cite{hom-kim2024}]\label{cong}
Two nonsingular matrices $A$ and $B$ in $PGL(3,\mathbb{F}_q)$  are congruent if and only if their corresponding curves $C_A$ and $C_B$ are projectively equivalent. More precisely, a matrix $T \in GL(3, \mathbb{F}_q)$ satisfies $B = \lambda T^* A T$ for some $\lambda \in \mathbb{F}_q^*$ if and only if $T$ induces a projective transformation from $C_B$ to $C_A$.
\end{lemma}
For a Hermitian-relative curve $C_A$, the curve $C_{A^*}$ is referred to as the mirror curve of $C_A$. It is important to note that $C_A$ is Hermitian if and only if   $C_A= C_{A^*} $. Let $\mathcal{L}_P(C)$ denote the tangent line to a curve $C$ at a point $P$. Recall that for $q > 4$, among nonsingular plane curves of degree $\sqrt{q}+1$, a Hermitian-relative curve $C$ is characterized by the following property: for any $\overline{\mathbb{F}}_q$-point $P$, the intersection multiplicity $i(C \cdot \mathcal{L}_P(C); P)$ is either $\sqrt{q}$ or $\sqrt{q}+1$ \cite{hom1987}.

\begin{proposition}[\cite{hom-kim2024}]\label{mirror}
For a mirror pair $(C_A, C_{A^*})$, the following properties hold:
\begin{enumerate}[{\rm (i)}]
\item $C_A(\mathbb{F}_q) = C_{A^*}(\mathbb{F}_q)$.
\item For $P \in C_A(\mathbb{F}_q)$, $P$ is an inflexion of $C_A$ if and only if it is an inflexion of $C_{A^*}$. In this case, their tangent lines coincide, i.e., $\mathcal{L}_P(C_A) = \mathcal{L}_P(C_{A^*})$.
\item If $P \in C_A(\mathbb{F}_q)$ is not an inflexion, there exists a unique point $P' \in C_A(\mathbb{F}_q) \setminus \{P\}$ such that the intersection divisor satisfies $C_A \cdot \mathcal{L}_P(C_A) = \sqrt{q}P + P'$. In this case, the intersection divisor for the mirror curve satisfies $C_{A^*} \cdot \mathcal{L}_{P'}(C_{A^*}) = \sqrt{q}P' + P$. In particular, $P'$ is also a non-inflexion of $C_{A^*}$.
\end{enumerate}
\end{proposition}

\begin{lemma}\label{mirror-inter}
Let $C_A$ be a Hermitian-relative curve that is not a Hermitian curve. Then the following hold:
\begin{enumerate}[{\rm (i)}]
\item $\sum_{P \in C_A(\mathbb{F}_q)} i(C_A \cdot C_{A^*}; P) \le \sum_{P \in C_A(\overline{\mathbb{F}}_q)} i(C_A \cdot C_{A^*}; P) = (\sqrt{q}+1)^2$.
\item If $P$ is an inflexion of $C_A$, then $i(C_A \cdot C_{A^*}; P) \ge \sqrt{q}+1$.
\item $N_q(C_A) \le (\sqrt{q}+1)^2$.
\end{enumerate}
\end{lemma}
\begin{proof}(i) Since $C_A(\mathbb{F}_q) \subseteq C_A(\overline{\mathbb{F}}_q)$, the inequality follows immediately. The equality is a direct consequence of Bézout's theorem.   (ii) If $P$ is an inflexion of $C_A$, then by Proposition \ref{mirror}, it is also an inflexion of $C_{A^*}$ and their tangent lines at $P$ coincide. Let $\mathcal{L} = \mathcal{L}_P(C_A) = \mathcal{L}_P(C_{A^*})$. By the characterization of inflexions on these curves, we have $i(C_A \cdot \mathcal{L}; P) = \sqrt{q}+1$ and $i(C_{A^*} \cdot \mathcal{L}; P) = \sqrt{q}+1$. This implies that $i(C_A \cdot C_{A^*}; P) \ge \sqrt{q}+1$. (iii) Since $C_A(\mathbb{F}_q) = C_{A^*}(\mathbb{F}_q)$ by Proposition \ref{mirror}, each $P \in C_A(\mathbb{F}_q)$ contributes at least $1$ to the intersection multiplicity $i(C_A \cdot C_{A^*}; P)$. Therefore, $N_q(C_A) \le \sum_{P \in C_A(\mathbb{F}_q)} i(C_A \cdot C_{A^*}; P)$, which is at most $(\sqrt{q}+1)^2$ by (1).
\end{proof}

\begin{theorem}[\cite{hom-kim2024}] \label{IqCge2}
Let $C$ be a Hermitian-relative curve. Suppose that $C$ has at least two rational inflexions. Then $C$ is projectively equivalent to the curve defined by
\begin{equation}\label{at_least_two_inflexions}
x^{\sqrt{q}}y + \eta xy^{\sqrt{q}} + z^{\sqrt{q}+1} = 0
\end{equation}
for some $\eta \in \mathbb{F}_q^{\ast}$. Conversely, any curve defined by \eqref{at_least_two_inflexions} has at least two rational inflexions. These curves are classified into the following three distinct types:
\begin{enumerate}[{\rm (i)}]
\item If $\eta = 1$, then the curve is Hermitian.
\item If $\eta^{\sqrt{q}+1} = 1$ with $\eta \neq 1$, then the curve \eqref{at_least_two_inflexions} has exactly $\sqrt{q}+1$ rational points, all of which are inflexions. In this case, the curve is projectively equivalent to
\begin{equation}\label{diagonaloftype1}
x^{\sqrt{q}+1} + y^{\sqrt{q}+1} + \eta' z^{\sqrt{q}+1} = 0
\end{equation}
for some $\eta' \in \mathbb{F}_q^{\ast} \setminus \mathbb{F}_{\sqrt{q}}^{\ast}$.
\item If $\eta^{\sqrt{q}+1} \neq 1$, then the curve \eqref{at_least_two_inflexions} has exactly $q+1$ rational points, of which exactly two are inflexions.
\end{enumerate}
Furthermore, this classification results in three distinct types of curves, as summarized in Table~\ref{Iqcge2}.
\end{theorem}
\begin{table}[h]

\centering

\caption{Classification of curves with $I_q(C) \ge 2$} \label{Iqcge2}

\begin{tabular}{|c|c|c|c|}

\hline

 & $N_q(C_A)$ & $I_q(C)$ & Number of classes \\

\hline

(a) & $\sqrt{q}^3 + 1$ & $\sqrt{q}^3 + 1$ & $1$ \\

\hline

(b) & $\sqrt{q} + 1$ & $\sqrt{q} + 1$ & $\sqrt{q}$ \\

\hline

(c) & $q + 1$ & $2$ & $\frac{1}{2}(\sqrt{q} + 1)(\sqrt{q} - 2)$ \\

\hline

\end{tabular}

\end{table}
The trace and norm maps from $\mathbb{F}_q$ to $\mathbb{F}_{\sqrt{q}}$ are denoted simply by $\mathrm{Tr}$ and $\mathrm{Nm}$, respectively. Note that $\mathrm{Tr}$ is an $\mathbb{F}_{\sqrt{q}}$-linear map, whereas $\mathrm{Nm}$ induces a multiplicative homomorphism from $\mathbb{F}_q^*$ to $\mathbb{F}_{\sqrt{q}}^*$.
\begin{lemma}[Hilbert's Theorem 90]\label{Hilbert}
The kernels of the trace and norm maps are given as follows:
\begin{enumerate}[{\rm (i)}]
\item $\rm{Ker}(\mathrm{Tr}) = \{ \alpha - \alpha^{\sqrt{q}} \mid \alpha \in \mathbb{F}_q \}$.
\item $\rm{Ker}(\mathrm{Nm}) = \{ \alpha^{1-\sqrt{q}} \mid \alpha \in \mathbb{F}_q^{\ast} \}$.
\end{enumerate}
\end{lemma}
\begin{proposition}[\cite{hom-kim2024}] \label{degsqrtq}
Let $\alpha \in \mathbb{F}_q^{\ast}$ and $\beta \in \mathbb{F}_q$,
and consider the equation in $X$:
\begin{equation}\label{sq10}
X^{\sqrt{q}} + \alpha X + \beta =0.
\end{equation}
\begin{enumerate}[{\rm(i)}]
 \item When $\mathrm{Nm} (\alpha) =1$, 
$(\ref{sq10})$ has a solution in $\mathbb{F}_q$ if and only if
$\beta =0$ or $\alpha = \beta^{1-\sqrt{q}}$ (which is equivalent to $\alpha^{\sqrt{q}}=\beta^{\sqrt{q}-1}$).
In this case, \eqref{sq10} possesses $\sqrt{q}$ distinct roots, all of which are $\mathbb{F}_q$-solutions.
 \item If $\mathrm{Nm}( \alpha) \neq 1$, then $(\ref{sq10})$
 has a unique $\mathbb{F}_q$-solution.
 \item
 In particular, when $\alpha=-1$, \eqref{sq10} has a solution if and only if $\mathrm{Tr}(\beta) =0$.
\end{enumerate}
\end{proposition}

In this section, we classify all curves for which every $\mathbb{F}_q$-rational point is an inflexion. Having already classified the cases where $I_q(C) \ge 2$ in Theorem \ref{IqCge2}, the remaining task is to examine the situation where $I_q(C) \le 1$. Under the condition $N_q(C) = I_q(C)$, and considering that $N_q(C) \ge 1$ for any Hermitian-relative curve, the only remaining possibility is the case $N_q(C) = I_q(C) = 1$. By proving the existence of such curves, the classification of all curves satisfying $N_q(C) = I_q(C)$ will be completed. In what follows, we show that curves of type $(1, 1)$ indeed exist.
\begin{theorem}\label{1pt}
For a Hermitian-relative curve $C$, the following statements are equivalent:
\begin{enumerate}[{\rm (a)}]
\item $N_q(C) = 1$.
\item $C$ is of type $(1,1)$.
\item $C$ is projectively equivalent to the curve $C_{A_\eta}$ defined by the equation
\begin{equation}\label{eq1pt}
x^{\sqrt{q}+1} + \eta y^{\sqrt{q}+1} + y^{\sqrt{q}} z + yz^{\sqrt{q}} = 0
\end{equation}
for some $\eta \in \mathbb{F}_q \setminus \mathbb{F}_{\sqrt{q}}$.
\end{enumerate}
Furthermore, any two curves satisfying these equivalent conditions are projectively equivalent.
\end{theorem}
\begin{proof}
By Proposition \ref{mirror}, if an $\mathbb{F}_q$-rational point $P$ is not an inflexion, then its tangent line must contain another $\mathbb{F}_q$-rational point $P' \in C(\mathbb{F}_q) \setminus \{P\}$. Thus, if $N_q(C) = 1$, the unique $\mathbb{F}_q$-point must be an inflexion, which shows that (a) implies (b). Conversely, (b) implies (a) by definition.

Now assume (c). 
Let $C$ be the curve defined by \eqref{eq1pt}. It is obvious that $P_0 = (0,0,1) \in C(\mathbb{F}_q)$. The tangent line to $C$ at $P_0$ is given by $y=0$, which contacts the curve with multiplicity $\sqrt{q}+1$. Thus, $P_0$ is an $\mathbb{F}_q$-rational inflexion of $C$. Now, let $(x_0 , y_0 , z_0) \in C(\mathbb{F}_q)$. Then we have
\begin{equation}\label{eq:trace_norm}x_0^{\sqrt{q}+1} + \eta y_0^{\sqrt{q}+1} + y_0^{\sqrt{q}} z_0 + y_0 z_0^{\sqrt{q}} = 0.
\end{equation}
Note that $x_0^{\sqrt{q}+1} = \mathrm{Nm}(x_0) \in \mathbb{F}_{\sqrt{q}}$ and $y_0^{\sqrt{q}} z_0 + y_0 z_0^{\sqrt{q}} = \mathrm{Tr}(y_0 z_0^{\sqrt{q}}) \in \mathbb{F}_{\sqrt{q}}$. Since $\eta \notin \mathbb{F}_{\sqrt{q}}$, the identity \eqref{eq:trace_norm} forces $y_0^{\sqrt{q}+1} = 0$, which implies $y_0 = 0$. Consequently, $x_0^{\sqrt{q}+1} = 0$, leading to $x_0 = 0$. Therefore, $P_0 = (0,0,1)$ is the unique $\mathbb{F}_q$-rational point of $C$, confirming that (a) holds.

Now we prove that (b) implies (c). 
Let $C$ be a Hermitian-relative curve of type $(1,1)$. We may assume that $P=(0:0:1)$ is the unique $\mathbb{F}_q$-rational point of $C$, which is also an inflexion, and that the line $y=0$ is the tangent line at $P$. Let $A=(a_{ij}) \in PGL(3,\mathbb{F}_q)$ be the matrix corresponding to $C$, i.e., $C=C_A$. Since $(0:0:1) \in C_A$, it follows that $a_{33}=0$. Because the line $y=0$ intersects the curve $C_A$ at $P$ with multiplicity $\sqrt{q}+1$, the polynomial
$$\begin{pmatrix} x^{\sqrt{q}} & 0 & 1 \end{pmatrix} A \begin{pmatrix} x \\ 0 \\ 1 \end{pmatrix} = a_{11}x^{\sqrt{q}+1} + a_{13}x^{\sqrt{q}} + a_{31}x$$
must be a nonzero constant multiple of $x^{\sqrt{q}+1}$. This implies that $a_{11} \ne 0$ and $a_{13}=a_{31}=0$. Since $A\in PGL(3,\mathbb{F}_q)$, we may assume that 
$A=\begin{pmatrix}1&a_{12}&0\\ a_{21}&a_{22}&a_{23}\\ 0&a_{32}&0\end{pmatrix}$. Note that $a_{23} \ne 0$ and $a_{32} \ne 0$ since $\det A \ne 0$.

Let $T_1 = \begin{pmatrix} 1&-a_{12}&0\\ 0&1&0\\0&0&1\end{pmatrix} \in GL(3,\mathbb{F}_q)$ and let$$A_1 = T_1^* A T_1 = \begin{pmatrix} 1 & 0 & 0 \\ a_{21}-a_{12}^{\sqrt{q}} & -a_{12}a_{21}+a_{22} & a_{23} \\ 0 & a_{32} & 0 \end{pmatrix}.$$
Next, let $T_2 = \begin{pmatrix} 1&0&0\\ 0&1&0\\0&0&a_{32}^{-\sqrt{q}}\end{pmatrix} \in GL(3,\mathbb{F}_q)$ and define
$$A_2 = T_2^* A_1 T_2 = \begin{pmatrix} 1 & 0 & 0 \\ b_{21} & b_{22} & b_{23} \\ 0 & 1 & 0 \end{pmatrix}$$
with $b_{21} = a_{21}-a_{12}^{\sqrt{q}}$, $b_{22} = a_{22}-a_{12}a_{21}$, and $b_{23} = a_{23}a_{32}^{-\sqrt{q}}$. Since $A$ and $A_2$ are congruent, the curve $C_A$ is projectively equivalent to $C_{A_2}$ by Lemma~\ref{cong}, where the latter is given by the equation
\begin{equation}
x^{\sqrt{q}+1} + b_{21}xy^{\sqrt{q}} + b_{22}y^{\sqrt{q}+1} + b_{23}y^{\sqrt{q}}z + yz^{\sqrt{q}} = 0.
\end{equation}
By the definitions of $T_1$ and $T_2$, $P=(0,0,1)$ is still the unique $\mathbb{F}_q$-rational point of $C_{A_2}$ with the tangent line $y=0$.

Since $P=(0,0,1)$ is the unique $\mathbb{F}_q$-rational point of $C_{A_2}$, the affine equation
\begin{equation}\label{y=1}
x^{\sqrt{q}+1} + b_{21}x + b_{22} + b_{23}z + z^{\sqrt{q}} = 0
\end{equation}
has no solutions in $\mathbb{F}_q^2$. For  $x_0 \in \mathbb{F}_q$, consider the following equation in $z$:
\begin{equation}\label{zdegsq}
z^{\sqrt{q}} + b_{23}z + (x_0^{\sqrt{q}+1} + b_{21}x_0 + b_{22}) = 0.
\end{equation}
Note that $b_{23} \ne 0$ because $\det A_2 \ne 0$. Since \eqref{zdegsq} has no solutions in $\mathbb{F}_q$, Proposition \ref{degsqrtq} implies that $\mathrm{Nm}(b_{23}) = 1$. Furthermore, as $\mathrm{Nm}(-b_{23}) =  \mathrm{Nm}(b_{23}) = 1$, it follows from Lemma \ref{Hilbert} that there exists an element $\rho \in \mathbb{F}_q^*$ such that \begin{equation} \label{b23} - b_{23} = \rho^{1-\sqrt{q}}.\end{equation}
Substituting this into \eqref{zdegsq}, we obtain
\begin{equation}\label{zdegsqrho}z^{\sqrt{q}} - \rho^{1-\sqrt{q}} z + (x_0^{\sqrt{q}+1} + b_{21} x_0 + b_{22}) = 0,
\end{equation}
which is equivalent to
\begin{equation}\label{zdegsqrhosq}
(\rho z)^{\sqrt{q}} - (\rho z) + \rho^{\sqrt{q}} (x_0^{\sqrt{q}+1} + b_{21} x_0 + b_{22}) = 0.
\end{equation}
Since \eqref{zdegsqrhosq} has no solutions in $\mathbb{F}_q$, Proposition \ref{degsqrtq} implies that $\mathrm{Tr}(\rho^{\sqrt{q}}(x_0^{\sqrt{q}+1} + b_{21} x_0 + b_{22})) \neq 0$ for all $x_0 \in \mathbb{F}_q$. Based on this, we establish the following:\\
\textbf{Claim.} $\rho + \rho^{\sqrt{q}} = 0$.\\
\textit{Proof of Claim.} Suppose, on the contrary, that $\rho + \rho^{\sqrt{q}} \neq 0$. Then it follows that 
\begin{equation}\label{trne0}
(\rho + \rho^{\sqrt{q}})^{\sqrt{q}} \mathrm{Tr}(\rho^{\sqrt{q}}(x_0^{\sqrt{q}+1} + b_{21} x_0 + b_{22})) \neq 0\end{equation}
On the other hand, a direct calculation shows that
\begin{equation}\begin{split}\label{rhoTr}
(\rho &+\rho^{\sqrt{q}})^{\sqrt{q}} \mathrm{Tr}(\rho^{\sqrt{q}}(x_0^{\sqrt{q}+1}+b_{21}x_0+b_{22})) \\
&= (\rho+\rho^{\sqrt{q}})^{\sqrt{q}} \left((\rho+\rho^{\sqrt{q}} )x_0^{\sqrt{q}+1}  + \rho b_{21}^{\sqrt{q}} x_0^{\sqrt{q}} +\rho^{\sqrt{q}} b_{21}x_0 +\rho b_{22}^{\sqrt{q}} +\rho^{\sqrt{q}} b_{22}\right)\\
&= \left((\rho+\rho^{\sqrt{q}} )x_0 + \rho b_{21} \right)^{\sqrt{q}+1} - \left(\rho^{\sqrt{q}+1} b_{21}^{\sqrt{q}+1} - (\rho+\rho^{\sqrt{q}})(\rho b_{22}^{\sqrt{q}}+\rho^{\sqrt{q}} b_{22})\right).
\end{split}\end{equation}
Since the norm map $\mathrm{Nm} : \mathbb{F}_q \to \mathbb{F}_{\sqrt{q}}$ is surjective and$$ \rho^{\sqrt{q}+1} b_{21}^{\sqrt{q}+1}-(\rho+\rho^{\sqrt{q}})(\rho b_{22}^{\sqrt{q}}+\rho^{\sqrt{q}} b_{22}) = \mathrm{Nm}(\rho b_{21}) - \mathrm{Tr}(\rho) \mathrm{Tr}(\rho b_{22}^{\sqrt{q}}) \in \mathbb{F}_{\sqrt{q}}, $$there exists an element $\zeta \in \mathbb{F}_q$ such that $\zeta^{\sqrt{q}+1} = \mathrm{Nm}(\rho b_{21}) - \mathrm{Tr}(\rho) \mathrm{Tr}(\rho b_{22}^{\sqrt{q}})$. Here, $\zeta$ may be zero. If we set $x_0 = \frac{\zeta - \rho b_{21}}{\rho + \rho^{\sqrt{q}}} \in \mathbb{F}_q$, then the expression in \eqref{rhoTr} vanishes, which contradicts \eqref{trne0}.
Thus, we must have $\rho + \rho^{\sqrt{q}} = 0$. This completes the proof of the Claim.

Now, returning to the proof of the Theorem, note that $\rho + \rho^{\sqrt{q}} = 0$ is equivalent to $\rho^{1-\sqrt{q}} = -1$. It then follows from \eqref{b23} that $b_{23} = -\rho^{1-\sqrt{q}} = 1$. Substituting this into \eqref{y=1}, we obtain
\begin{equation}\label{b23=1}
x^{\sqrt{q}+1} + b_{21}x + b_{22} + z + z^{\sqrt{q}} = 0.
\end{equation}
If $b_{21}\ne 0$, substituting $x= - \frac{b_{22}}{b_{21}}$, we have $z+z^{\sqrt{q}} = - (\frac{b_{22}}{b_{21}})^{\sqrt{q}+1}$. Since $- (\frac{b_{22}}{b_{21}})^{\sqrt{q}+1}\in \mathbb{F}_{\sqrt{q}}$ and $\mathrm{Tr}:\mathbb{F}_{{q}}\rightarrow\mathbb{F}_{\sqrt{q}}$ is surjective, there exists an 
 $\mathbb{F}_q$-solution for $z$. This implies the existence of an $\mathbb{F}_q$-rational point on \eqref{b23=1}, which contrads our assumption. Thus we conclude that $b_{21} =0$. So \eqref{b23=1} is reduced to 
\begin{equation}\label{b21=0}
x^{\sqrt{q}+1}  + b_{22} +z  + z^{\sqrt{q}} = 0
\end{equation}
If $b_{22} \in \mathbb{F}_{\sqrt{q}}$, then for any $x_0 \in \mathbb{F}_q$, the value $-(x_0^{\sqrt{q}+1} + b_{22})$ belongs to $\mathbb{F}_{\sqrt{q}}$. By the surjectivity of the trace map, there exists an $\mathbb{F}_q$-solution to \eqref{b21=0}, which contradicts the uniqueness of the point $P$. Thus, $b_{22}$ must be an element of $\mathbb{F}_q \setminus \mathbb{F}_{\sqrt{q}}$. By setting $b_{22} = \eta \in \mathbb{F}_q \setminus \mathbb{F}_{\sqrt{q}}$, we conclude that $A$ is projectively equivalent to
$$ A_2 = \begin{pmatrix} 1 & 0 & 0 \\ 0 & \eta & 1 \\ 0 & 1 & 0 \end{pmatrix}. $$
This matrix corresponds to the curve
$$ C_{A_2} : x^{\sqrt{q}+1} + \eta y^{\sqrt{q}+1} + y^{\sqrt{q}} z + y z^{\sqrt{q}} = 0. $$

Now that the equivalence has been established, we conclude by showing the uniqueness of the curve up to projective equivalence. Let
$$ A_3 = \begin{pmatrix} 1 & 0 & 0 \\ 0 & \eta' & 1 \\ 0 & 1 & 0 \end{pmatrix} \quad \text{with } \eta' \in \mathbb{F}_q \setminus \mathbb{F}_{\sqrt{q}}, $$
and let $C_{A_3}$ be its corresponding curve:
$$ C_3 :\ x^{\sqrt{q}+1} + \eta' y^{\sqrt{q}+1} + y^{\sqrt{q}} z + y z^{\sqrt{q}} = 0. $$
Since $\mathbb{F}_q$ is a two-dimensional vector space over $\mathbb{F}_{\sqrt{q}}$, the set $\{1, \eta\}$ forms a basis for $\mathbb{F}_q$ over $\mathbb{F}_{\sqrt{q}}$. Hence there exist $\alpha, \beta \in \mathbb{F}_{\sqrt{q}}$ with $\alpha \neq 0$ such that $\eta' = \alpha \eta + \beta$. Recall that the norm and trace maps from $\mathbb{F}_q$ to $\mathbb{F}_{\sqrt{q}}$ are surjective. Therefore, there exist $\zeta, \xi \in \mathbb{F}_q$ such that $\zeta^{\sqrt{q}+1} = \alpha$ and $\xi^{\sqrt{q}} + \xi = \beta$. Note that $\alpha \neq 0$ implies $\zeta \neq 0$. Now, let $T$ be the nonsingular matrix defined by
$$ T = \begin{pmatrix} 1 & 0 & 0 \\ 0 & \zeta & 0 \\ 0 & \zeta^{-\sqrt{q}} \xi & \zeta^{-\sqrt{q}} \end{pmatrix}. $$
A direct computation yields:
\begin{equation*}\begin{split}
T^* A_2 T &=
\begin{pmatrix} 1&0&0\\0&\zeta^{\sqrt{q}}&\ \zeta^{-1}\xi^{\sqrt{q}} \\0&0&\zeta^{-1} \end{pmatrix}
\begin{pmatrix} 1&0&0\\0&\eta&1\\0&1&0 \end{pmatrix}
\begin{pmatrix} 1&0&0 \\ 0&\zeta&0 \\ 0&\zeta^{-\sqrt{q}}\xi &\ \zeta^{-\sqrt{q}} \end{pmatrix} \\
&=\begin{pmatrix} 1&0&0\\0&\ \zeta^{\sqrt{q}+1}\eta +\xi^{\sqrt{q}}+\xi&1\\0&\ \ 1 &0 \end{pmatrix}
=\begin{pmatrix} 1&0&0\\0&\eta'&1\\0&1&0 \end{pmatrix}= A_3.
\end{split}\end{equation*}
This shows that the matrices $A_2$ and $A_3$ are congruent. Consequently, the curves $C_{A_2}$ and $C_{A_3}$ are projectively equivalent, which completes the proof of uniqueness.

\end{proof}

As stated at the beginning of this section, having established the existence of curves of type $(1, 1)$ in Theorem \ref{1pt}, we now obtain the following theorem, which fulfills the primary objective of this section.
\begin{theorem}\label{same}There exist Hermitian-relative curves of types $(q\sqrt{q}+1, q\sqrt{q}+1)$, $(\sqrt{q}+1, \sqrt{q}+1)$, and $(1, 1)$. Moreover, if a Hermitian-relative curve $C$ satisfies $N_q(C) = I_q(C)$, then its type must be one of these three cases.
\end{theorem}
\section{A Restricted Model for $N_q(C) \neq I_q(C)$}
In the preceding section, we classified curves where every $\mathbb{F}_q$-rational point is an inflexion. We now turn our attention to curves that possess at least one $\mathbb{F}_q$-rational point that is not an inflexion. We begin by deriving a simplified equation for such curves.

As established in Proposition~\ref{mirror}, the intersection of a Hermitian-relative curve $C$ and its tangent line $\mathcal{L}_P C$ at an $\mathbb{F}_q$-rational point $P$ is given by $C \cdot \mathcal{L}_P C = \sqrt{q}P + P'$. Based on this, we define a map $\varphi : C(\mathbb{F}_q) \to C(\mathbb{F}_q)$ by $\varphi(P) = P'$. When the curve is specified as $C_A$ and it is necessary to emphasize its dependence on the matrix $A$, we shall denote this map by $\varphi_A$ instead of $\varphi$. It follows that $P$ is an inflexion point if and only if $\varphi(P) = P$.
\begin{proposition}\label{ext_varphi}
Let $A \in GL(3, \mathbb{F}_q)$. For any point $P \in C_A$, not necessarily an $\mathbb{F}_q$-rational point, there exists a unique point $P' \in C_A$ such that $\mathcal{L}_P C_A \cdot C_A = \sqrt{q}P + P'$. The map $\varphi_A : C_A \to C_A$, defined by $\varphi_A(P) = P'$, can be expressed explicitly as 
\[
\varphi_A(P) = A^{-1}A^* P^{(q)},
\]
where $((x, y, z)^t)^{(q)} = (x^q, y^q, z^q)^t$. Note that if $P$ is an $\mathbb{F}_q$-rational point, then $P^{(q)} = P$, so this expression is consistent with the definition of $\varphi$ in Proposition~\ref{mirror}.
\end{proposition}
\begin{proof}
Let $C_J$ be the Hermitian curve defined by $x^{\sqrt{q}+1} + y^{\sqrt{q}}z + yz^{\sqrt{q}} = 0$, which corresponds to the matrix $J = \begin{pmatrix} 1 & 0 & 0 \\ 0 & 0 & 1 \\ 0 & 1 & 0 \end{pmatrix}$. Since $J = J^*$, the curve $C_J$ is Hermitian. Since $C_J$ and $C_A$ are projectively equivalent over $\overline{\mathbb{F}}_q$~\cite{par1986}, there exists a matrix $T \in GL(3, \overline{\mathbb{F}}_q)$ such that $T^* J T = A$. It is well known that the map $\varphi_J$ for $C_J$ is given by $\varphi_J(P) = P^{(q)}$ (see, e.g., \cite[Lemma 3.1]{hom2006}). Since the projective transformation $T$ induces an isomorphism between $C_A$ and $C_J$ while preserving tangency, the map $\varphi_A$ can be expressed via conjugation as:
\begin{equation}\label{varphiAP}
\varphi_A(P) = T^{-1} \varphi_J(T P) = T^{-1} (T P)^{(q)} = T^{-1} T^{(q)} P^{(q)}.
\end{equation}
The relation $T^* J T = A$ implies $T^{-1} = A^{-1} T^* J$. Taking the Hermitian adjoint of $A = T^* J T$ yields $A^* = T^* J T^{(q)}$, or equivalently $T^{(q)} = J (T^*)^{-1} A^*$. Given $J^2 = I$, their product simplifies directly to:

$$T^{-1} T^{(q)} = (A^{-1} T^* J)(J (T^*)^{-1} A^*) = A^{-1} A^*.$$
Thus, we obtain the explicit expression $\varphi_A(P) = A^{-1} A^* P^{(q)}$.
\end{proof}
Next proposition characterizes the inflexions on $C_A$.
\begin{proposition}\label{eigenvector} Let $A \in GL(3, \mathbb{F}_{q})$ and $P \in C_A(\mathbb{F}_{q})$. Then $P$ is an inflexion of $C_A$ if and only if $P$ is an eigenvector of the matrix $A^{-1}A^*$.
\end{proposition}
\begin{proof}
By the definition of the map $\varphi_A$, a point $P$ is an inflexion if and only if $\varphi_A(P) = P$ in $\mathbb{P}^2$. For $P \in C_A( \mathbb{F}_q)$, we have $P^{(q)} = P$, and thus the condition $\varphi_A(P) = P$ is equivalent to $A^{-1}A^* P = \lambda P$ for some $\lambda \in \mathbb{F}_q^*$. This is precisely the condition that $P$ is an eigenvector of $A^{-1}A^*$, which completes the proof.
\end{proof}
\begin{proposition}\label{Cab}
Let $C$ be a Hermitian-relative curve. Then $C$ possesses an $\mathbb{F}_q$-rational point that is not an inflexion if and only if $C$ is projectively equivalent to the curve defined by
\begin{equation}\label{eqCab}
C_{a,b} : x^{\sqrt{q}} z + x y^{\sqrt{q}} + b x z^{\sqrt{q}} + a y z^{\sqrt{q}} = 0,
\end{equation}
where $a, b \in \mathbb{F}_q$ and $a \neq 0$. This curve corresponds to the matrix $\begin{pmatrix}0&0&1\\ 1&0&0\\ b&a&0 \end{pmatrix}$.
\end{proposition}
\begin{proof}
Note that the points $(1,0,0)^t$, $(0,1,0)^t$, and $(0,0,1)^t$ lie on $C_{a,b}$. At the point $P=(1,0,0)^t$, the tangent line $\mathcal{L}_P$ meets the curve at $(0,1,0)^t$ as well. This implies that the intersection multiplicity of $\mathcal{L}_PC_{a,b}$ and $C$ at $P$ is less than the order required for an inflexion, confirming that $P$ is a non-inflexion point.

Conversely, suppose that a Hermitian-relative curve $C$ possesses an $\mathbb{F}_q$-rational point $P$ that is not an inflexion. The fact that $P$ is not an inflexion implies $\varphi(P) \neq P$. From the intersection relations $\mathcal{L}_P C \cdot C = \sqrt{q}P + \varphi(P)$ and $\mathcal{L}_{\varphi(P)} C \cdot C = \sqrt{q}\varphi(P) + \varphi^2(P)$, it follows that the two tangent lines $\mathcal{L}_P C$ and $\mathcal{L}_{\varphi(P)} C$ are distinct and intersect at the point $\varphi(P)$. Consequently, the three points $P$, $\varphi(P)$, and $\varphi^2(P)$ are not collinear.  Since the three points $P, \varphi(P)$, and $\varphi^2(P)$ are not collinear, we may assume, without loss of generality, that $P = (1, 0, 0)^t$, $\varphi(P) = (0, 1, 0)^t$, and $\varphi^2(P) = (0, 0, 1)^t$ by applying a suitable projective transformation. 
Let $A = (a_{ij}) \in GL(3, \mathbb{F}_q)$ be the matrix corresponding to $C$; that is, $C = C_A$. Since the tangent line $\mathcal{L}_P C$ passes through both $P = (1, 0, 0)^t$ and $\varphi(P) = (0, 1, 0)^t$, its equation is given by $z = 0$. Similarly, $\mathcal{L}_{\varphi(P)} C$ passes through $\varphi(P) = (0, 1, 0)^t$ and $\varphi^2(P) = (0, 0, 1)^t$, and hence its equation is $x = 0$. On the other hand, the coefficients of the tangent lines $\mathcal{L}_P C$ and $\mathcal{L}_{\varphi(P)} C$ are given by $[a_{11}, a_{12}, a_{13}]$ and $[a_{21}, a_{22}, a_{23}]$, respectively. Comparing these with the equations $z = 0$ and $x = 0$, we obtain$$a_{11} = a_{12} = a_{22} = a_{23} = 0.$$ Furthermore, since the curve $C_A$ passes through $(0, 0, 1)^t$, it follows that $a_{33} = 0$. Since $A$ is non-singular, the entry $a_{13}$ must be non-zero (otherwise, the first row would be zero). Thus, we may divide the matrix by $a_{13}$ and assume, without loss of generality, that $a_{13} = 1$. Consequently, the matrix $A$ takes the form
$$A = \begin{pmatrix} 0 & 0 & 1 \\ a_{21} & 0 & 0 \\ a_{31} & a_{32} & 0 \end{pmatrix}.$$
Since the determinant of $A$, given by $-a_{21} a_{32}$, is non-zero, it follows that $a_{21} \neq 0$ and $a_{32} \neq 0$. The corresponding equation for the curve is then
$$ x^{\sqrt{q}} z + a_{21} x y^{\sqrt{q}} + a_{31} x z^{\sqrt{q}} + a_{32} y z^{\sqrt{q}} = 0.$$
If we let $T = \text{diag}(1, a_{21}^{-\sqrt{q}}, 1)$, then 
\[
T^*AT = \begin{pmatrix} 0 & 0 & 1 \\ 1 & 0 & 0 \\ a_{31} & a_{32} a_{21}^{-\sqrt{q}} & 0 \end{pmatrix}.
\]
Upon substituting $b = a_{31}$ and $a = a_{32} a_{21}^{-\sqrt{q}}$, the matrix $A$ is transformed into the canonical form associated with the curve $C_{a,b}$:
\[
\begin{pmatrix} 0 & 0 & 1 \\ 1 & 0 & 0 \\ b & a & 0 \end{pmatrix}.
\]
Note that $a \neq 0$ since $a_{21}$ and $a_{32}$ are non-zero.
\end{proof}
We first consider the curves with $b = 0$, namely the curves
\begin{equation*}C_{a,0} : x^{\sqrt{q}}z + xy^{\sqrt{q}} + ayz^{\sqrt{q}} = 0,
\end{equation*}where $a \in \mathbb{F}_q^*$.
\begin{proposition}\label{N_qCa0}The number of rational points on $C_{a,0}$ is given as follows:
\begin{enumerate}[{\rm (i)}]
\item If $\gcd(\sqrt{q}+1, 3) = 1$, then $N_q(C_{a,0}) = q+1$.
\item If $\gcd(\sqrt{q}+1, 3) = 3$, then
\begin{equation*}N_q(C_{a,0}) =
\begin{cases}q + 2\sqrt{q} + 1 & \text{if } a^{(q-1)/3} = 1, \\ q - \sqrt{q} + 1 & \text{if } a^{(q-1)/3} \neq 1.
\end{cases}\end{equation*}\end{enumerate}\end{proposition}
\begin{proof}
The points on the curve $C_{a,0}$ lying on the line $z=0$ are $(1,0,0)$ and $(0,1,0)$. For the points in the affine plane ($z \neq 0$), we consider the affine equation:
\begin{equation}\label{Ca0aff}
x^{\sqrt{q}} + x y^{\sqrt{q}} + ay = 0.
\end{equation}
We now determine the number of solutions for $x$ by applying Proposition~\ref{degsqrtq} for each $y \in \mathbb{F}_q$.
If $\mathrm{Nm}(y) \neq 1$, the equation \eqref{Ca0aff} has exactly one solution for $x$. Since there are $q - (\sqrt{q}+1)$ such values of $y$ in $\mathbb{F}_q$ (including $y=0$, where $\mathrm{Nm}(0)=0 \neq 1$), we obtain $(q - \sqrt{q} - 1)$ points in the affine plane. Including the two points at infinity, we have a total of $q - \sqrt{q} + 1$ points so far. Next, we consider the case $\mathrm{Nm}(y) = 1$. There are exactly $\sqrt{q}+1$ such elements in $\mathbb{F}_q$. For these values, the equation \eqref{Ca0aff} has either $\sqrt{q}$ solutions or no solution. By Proposition~\ref{degsqrtq}, it has $\sqrt{q}$ solutions if and only if the constant term $ay$ satisfies $ay = 0$ or $(ay)^{1-\sqrt{q}} = y^{\sqrt{q}}$. Since $\mathrm{Nm}(y)=1$ implies $ay \neq 0$, the condition is equivalent to:\begin{equation}\label{cond_i}(ay)^{1-\sqrt{q}} = y^{\sqrt{q}}.\end{equation}Let $i$ be the number of elements $y \in \mathbb{F}_q$ satisfying $\mathrm{Nm}(y) = 1$ and \eqref{cond_i}. Then the total number of rational points on $C_{a,0}$ is:\begin{equation*}N_q(C_{a,0}) = (q - \sqrt{q} + 1) + i\sqrt{q}.\end{equation*}
The equation $(ay)^{1-\sqrt{q}} = y^{\sqrt{q}}$ is equivalent to $(ay)^{\sqrt{q}-1} = y$ (by raising both sides to the $\sqrt{q}$-th power), or simply, $(ay)^{\sqrt{q}-1} \cdot y^{-1} = 1$. By rearranging the terms, we obtain:
\begin{equation}\label{y_eq}
y^{\sqrt{q}-2} = a^{1-\sqrt{q}}.
\end{equation}
Let $\omega$ be a primitive element of $\mathbb{F}_q$. Since $\mathrm{Nm}(y) = 1$, we can represent $y$ as $y = \omega^{(\sqrt{q}-1)k}$ for some $k \in \{0, 1, \dots, \sqrt{q}\}$. Let $a = \omega^j$. Substituting these into \eqref{y_eq}, we have:
\begin{equation*}
\left( \omega^{(\sqrt{q}-1)k} \right)^{\sqrt{q}-2} = \left( \omega^j \right)^{1-\sqrt{q}},
\end{equation*}
which is equivalent to the following congruence modulo $q-1$:
\begin{equation*}
(\sqrt{q}-2)(\sqrt{q}-1)k \equiv j(1-\sqrt{q}) \pmod{q-1}.
\end{equation*}
Dividing both sides by $\sqrt{q}-1$ (noting that $q-1 = (\sqrt{q}-1)(\sqrt{q}+1)$), we obtain:
\begin{equation*}
(\sqrt{q}-2)k \equiv -j \pmod{\sqrt{q}+1}.
\end{equation*}
Since $\sqrt{q}-2 \equiv -3 \pmod{\sqrt{q}+1}$, the congruence simplifies further to:
\begin{equation}\label{final_cong}
-3k \equiv -j \pmod{\sqrt{q}+1}, \quad \text{or} \quad 3k \equiv j \pmod{\sqrt{q}+1}.
\end{equation}
The number of solutions for $k$ in the congruence $3k \equiv j \pmod{\sqrt{q}+1}$ depends on the value of $ \gcd(\sqrt{q}+1, 3)$.
\begin{enumerate}[{\rm (i)}]
\item If $ \gcd(\sqrt{q}+1, 3) = 1$, the congruence has a unique solution for $k \in \{0, 1, \dots, \sqrt{q}\}$ for any $j$. Thus, $i = 1$, which yields$$ N_q(C_{a,0}) = (q - \sqrt{q} + 1) + 1 \cdot \sqrt{q} = q + 1. $$
\item If $ \gcd(\sqrt{q}+1, 3) = 3$, the congruence has three distinct solutions if $j$ is a multiple of $3$, and no solution otherwise. Since $q-1$ is a multiple of $3$ when $\sqrt{q}+1$ is a multiple of $3$, the condition that $j$ is a multiple of $3$ is equivalent to 
\[ a^{\frac{q-1}{3}} = (\omega^j)^{\frac{q-1}{3}} = (\omega^{q-1})^{\frac{j}{3}} = 1. \]
Consequently, if $a^{\frac{q-1}{3}} = 1$, we have $i = 3$ and 
\[ N_q(C_{a,0}) = (q - \sqrt{q} + 1) + 3\sqrt{q} = q + 2\sqrt{q} + 1. \]
If $a^{\frac{q-1}{3}} \neq 1$, then $i = 0$ and 
\[ N_q(C_{a,0}) = q - \sqrt{q} + 1. \]
\end{enumerate}
This completes the proof of Proposition \ref{N_qCa0}.
\end{proof}
%
\begin{proposition}\label{I_qCa0}
The number $I_q(C_{a,0})$ of rational inflexions of the curve $C_{a,0}$ is determined by the value of $\sqrt{q}$ modulo $3$ as follows:
\begin{enumerate}[{\rm (i)}]
\item If $\sqrt{q} \equiv 0 \pmod{3}$, then $I_q(C_{a,0}) = 1$.
\item If $\sqrt{q} \equiv 1 \pmod{3}$, then $I_q(C_{a,0}) = 2$.
\item If $\sqrt{q} \equiv 2 \pmod{3}$, then $I_q(C_{a,0}) = 0$.
\end{enumerate}
\end{proposition}
\begin{proof}Since the two points $(1, 0, 0)^t$ and $(0, 1, 0)^t$ on the line at infinity $z=0$ are not inflexions, it suffices to search for inflexions within the affine plane $z \neq 0$.

The matrix associated with the curve $C_{a,0}$ is given by
$$A = \begin{pmatrix} 0 & 0 & 1 \\ 1 & 0 & 0 \\ 0 & a & 0 \end{pmatrix}.$$
To apply Proposition~\ref{eigenvector}, we consider the matrix $U = A^{-1} A^*$.  A straightforward calculation yields
$$U = \begin{pmatrix} 0 & 0 & a^{\sqrt{q}} \\ 1/a & 0 & 0 \\ 0 & 1 & 0 \end{pmatrix},$$
and its characteristic polynomial is
\begin{equation}\label{charpoly}
P_U(\lambda) = \det(\lambda I - U) = \lambda^3 - a^{\sqrt{q}-1}.
\end{equation}
For an eigenvalue $\lambda$ of $U$,   a straightforward calculation shows that the coordinates of the eigenvector satisfy
$$(x_0, y_0, 1)^t = \left( \frac{a^{\sqrt{q}}}{\lambda}, \lambda, 1 \right)^t.$$
Substituting these coordinates into the left-hand side of the affine equation \eqref{Ca0aff}, and utilizing the relation $\lambda^3 = a^{\sqrt{q}-1}$, the expression simplifies as follows:
\begin{equation}\label{Ca0lambda}
x_0^{\sqrt{q}} + x_0 y_0^{\sqrt{q}} + a y_0 = \frac{a}{\lambda^{\sqrt{q}}} \left( \lambda^{2(\sqrt{q}+1)} + \lambda^{\sqrt{q}+1} + 1 \right). 
\end{equation}
According to Proposition~\ref{eigenvector}, a point $(x_0, y_0, 1)^t$ on the curve is an inflexion if and only if it is an eigenvector of $U = A^{-1} A^*$. 
Hence the expression in \eqref{Ca0lambda} must vanish. Since $a / \lambda^{\sqrt{q}} \neq 0$ in \eqref{Ca0lambda}, we can divide by this factor to obtain the second equation in the system below. Consequently, the number of rational inflexions is equal to the number of solutions for $\lambda$ in the following system:
\begin{equation}\label{eqinfl}
\begin{cases}\lambda^3 - a^{\sqrt{q}-1} = 0, \\\lambda^{2(\sqrt{q}+1)} + \lambda^{\sqrt{q}+1} + 1 = 0.\end{cases}\end{equation}
(i) The case $\sqrt{q} \equiv 0 \pmod{3}$. In this case, the characteristic of $\mathbb{F}_q$ is $3$. Since the Frobenius map $\sigma: u \mapsto u^3$ is an automorphism of $\mathbb{F}_q$, the equation $\lambda^3 = a^{\sqrt{q}-1}$ has a unique solution in $\mathbb{F}_q$, given by $\lambda = (a^{\sqrt{q}-1})^{q/3}$. Substituting this into the second equation of \eqref{eqinfl}, we obtain:
\begin{equation*}
\left( a^{\frac{q(\sqrt{q}-1)}{3}} \right)^{2(\sqrt{q}+1)} + \left( a^{\frac{q(\sqrt{q}-1)}{3}} \right)^{\sqrt{q}+1} + 1 =  a^{(q-1)\frac{2q }{3}}  +  a^{(q-1)\frac{q }{3}}  + 1 =  0.
\end{equation*}
Therefore, the system \eqref{eqinfl} has a unique solution for $\lambda$, which implies that the curve has a unique rational inflexion.

\noindent (ii) The case $\sqrt{q} \equiv 1 \pmod{3}$. In this case, $3$ divides $\sqrt{q}-1$, consequently, 3 also divides $q-1$. The equation $\lambda^3 = a^{\sqrt{q}-1} $ possesses three distinct roots in $\mathbb{F}_q$, given by:
\begin{equation}\label{3sol} a^{\frac{  \sqrt{q}-1 }{3}}, \quad \zeta  a^{\frac{  \sqrt{q}-1 }{3}}, \quad \zeta^2  a^{\frac{  \sqrt{q}-1 }{3}},
\end{equation}
where $\zeta$ is a primitive cube root of unity (i.e., a root of $X^3=1$). Since $3 \mid (\sqrt{q}-1)$, it follows that $\zeta$ lies in the subfield $\mathbb{F}_{\sqrt{q}}^* \subset \mathbb{F}_q^*$. 

Let $\rho = \mathrm{Nm}\left( a^{\frac{\sqrt{q}-1}{3}} \right)$. Note that $\rho = a^{\frac{q-1}{3}}$, which implies $\rho^3 = 1$. Then the norms of the other two roots are given by:$$\mathrm{Nm}\left( \zeta a^{\frac{\sqrt{q}-1}{3}} \right) = \zeta^{\sqrt{q}+1} \rho = \zeta^2 \rho, \quad \text{and} \quad
\mathrm{Nm}\left( \zeta^2 a^{\frac{\sqrt{q}-1}{3}} \right) = \zeta^{2(\sqrt{q}+1)} \rho = \zeta \rho.$$

Thus, the set of norms of the three roots, $\{\rho, \zeta^2 \rho, \zeta \rho\}$, consists of distinct cube roots of unity, and is therefore a permutation of $\{1, \zeta, \zeta^2\}$. On the other hand, the left-hand side of the second equation in \eqref{eqinfl} can be written as $\mathrm{Nm}(\lambda)^2 + \mathrm{Nm}(\lambda) + 1$.  Among the three possible values for the norm, $\zeta$ and $\zeta^2$ satisfy this equation, whereas $1$ does not (since the characteristic is not $3$). Consequently, there are exactly two solutions for $\lambda$, which implies that the curve has exactly two rational inflexions.

\noindent (iii) The case $\sqrt{q} \equiv 2 \pmod{3}$. In this case, we have $3 \mid (\sqrt{q}+1)$, which implies $\gcd(3, \sqrt{q}-1) = 1$. Suppose that $\lambda_0$ be a solution to the equation $\lambda_0^3 = a^{\sqrt{q}-1}$ in $\mathbb{F}_q$. To determine the norm of this solution, we observe that:$$\left( \mathrm{Nm}(\lambda_0) \right)^3 = (\lambda_0^3)^{\sqrt{q}+1} = (a^{\sqrt{q}-1})^{\sqrt{q}+1} = a^{q-1} = 1.$$Since $\gcd(3, \sqrt{q}-1) = 1$, the only cube root of unity in the subfield $\mathbb{F}_{\sqrt{q}}$ is $1$ itself. Therefore, it follows that:
$$\mathrm{Nm}(\lambda_0) = \lambda_0^{\sqrt{q}+1} = 1.$$
Substituting this into the left-hand side of the second equation in \eqref{eqinfl}, we obtain:
$$\lambda_0^{2(\sqrt{q}+1)} + \lambda_0^{\sqrt{q}+1} + 1 = 1^2 + 1 + 1 = 3 \neq 0,$$
as the characteristic of the field is not $3$. Consequently, there are no solutions to \eqref{eqinfl} for $\lambda_0$ in this case, meaning that the curve has no rational inflexions.

\end{proof}
From Propositions~\ref{N_qCa0} and \ref{I_qCa0}, we obtain the following theorem.
\begin{theorem}\label{Ca0}
For the curve $C_{a,0}$, the following hold:
\begin{enumerate}[{\rm (i)}]
\item If $\sqrt{q}\equiv 0 \pmod{3}$, then $C_{a,0}$ is of type $(q+1, 1)$.
\item If $\sqrt{q}\equiv 1 \pmod{3}$, then $C_{a,0}$ is of type $(q+1, 2)$.
\item If $\sqrt{q}\equiv 2 \pmod{3}$, then $C_{a,0}$ is of type:
\begin{itemize}
\item $(q+2\sqrt{q}+1, 0)$ if $a = \omega^{3j}$ for some $0 \le j < \frac{q-1}{3}$,
\item $(q-\sqrt{q}+1, 0)$ otherwise.
\end{itemize}Here, $\omega$ is a primitive element of $\mathbb{F}_q$.
\end{enumerate}
\end{theorem}
\begin{proof}
The proof follows directly from the results established in Proposition~\ref{N_qCa0} and Proposition~\ref{I_qCa0}.
\end{proof}

\section{Complete Characterization of the General Case}
In Section 3, we studied curves with $N_q(C) \neq I_q(C)$ using the simple equation \eqref{eqCab} and found all types $(N_q(C), I_q(C))$ for the case $b = 0$. In this section, we consider the case $b \neq 0$. For these curves, we transform the equation into a form with only one parameter, which allows for a systematic analysis of their properties.
\begin{proposition}\label{prop:Cc}
Any Hermitian-relative curve $C_{a,b}$ with $a, b \in \mathbb{F}_q^*$ defined by
\begin{equation*}
C_{a,b} : x^{\sqrt{q}} z + x y^{\sqrt{q}} + b x z^{\sqrt{q}} + a y z^{\sqrt{q}} = 0
\end{equation*}
is projectively equivalent to a curve $C_c$ for some $c \in \mathbb{F}_q^*$, where $C_c$ is given by
\begin{equation}\label{eq:Cc}
C_c : x^{\sqrt{q}}z - cxy^{\sqrt{q}} + cxz^{\sqrt{q}} - yz^{\sqrt{q}} = 0.
\end{equation}
The curve $C_c$ corresponds to the matrix $M_c = \begin{pmatrix} 0 & 0 & 1 \\ -c & 0 & 0 \\ c & -1 & 0 \end{pmatrix}$. Conversely, any curve $C_c$ with $c \neq 0$ is projectively equivalent to $C_{a,b}$ for some $a, b \in \mathbb{F}_q^*$.
\end{proposition}
\begin{proof}
Note that $A_{a,b} = \begin{pmatrix} 0 & 0 & 1 \\ 1 & 0 & 0 \\ b & a & 0 \end{pmatrix}$ is the matrix corresponding to $C_{a,b}$. Let $T$ be the diagonal matrix $T = \mathrm{diag}(ba^{\sqrt{q}}, -b^{\sqrt{q}}, 1)$. A direct calculation yields
$$T^* A_{a,b} T = \begin{pmatrix} 0 & 0 & ab^{\sqrt{q}} \\ -b^2a^{\sqrt{q}} & 0 & 0 \\ b^2a^{\sqrt{q}} & -ab^{\sqrt{q}} & 0 \end{pmatrix} = ab^{\sqrt{q}} \begin{pmatrix} 0 & 0 & 1 \\ -c & 0 & 0 \\ c & -1 & 0 
\end{pmatrix},$$
where $c = \frac{b^2 a^{\sqrt{q}}}{a b^{\sqrt{q}}} = b (\frac{a}{b})^{\sqrt{q}-1}$. Since $ab^{\sqrt{q}}$ is a non-zero scalar, this matrix is equivalent to $M_c$ in $PGL(3, \mathbb{F}_q)$, which proves the projective equivalence of $C_{a,b}$ and $C_c$. Conversely, for any $c \in \mathbb{F}_q^*$, we can choose $a = b = c$. In this case, the parameter $c$ in our transformation becomes $c (\frac{c}{c})^{\sqrt{q}-1} = c$. Thus, $C_{c,c}$ is projectively equivalent to $C_c$, showing that every $C_c$ can be realized as a curve of the form $C_{a,b}$.
\end{proof}
 Since  $c \in \mathbb{F}_q^*$ it suffices to study the $q-1$ matrices and their corresponding $q-1$ curves. In this section, we fix the notation $C_c$ for the curve defined in \eqref{eq:Cc} and denote its corresponding matrix by $M:=M_c$. To simplify the exposition, we shall write $M$ instead of $M_c$ whenever there is no risk of confusion. Also, when we need, we use $C_{M_c}$ instead of $C_c$. 

Recall from Proposition \ref{ext_varphi}  that the matrix $U := M^{-1}M^*$ plays a leading role in handling the general case. In the following lemma, we specify this matrix and its algebraic properties for the curve $C_c$.
\begin{lemma}\label{prop:basic}
Let $C_c$ be the curve defined in \eqref{eq:Cc} and $M$ be its associated matrix. The following properties hold:
\begin{enumerate}[{\rm (i)}]
\item $U:=M^{-1}M^* = 
\begin{pmatrix} 0&0&c^{-1}\\ -1&0&1\\ 0&-c^{\sqrt{q}}&c^{\sqrt{q}} \end{pmatrix}
$. 
\item $P(\lambda):=P_U(\lambda) = \lambda^3 - c^{\sqrt{q}} \lambda^2 +  c^{\sqrt{q}} \lambda - c^{\sqrt{q}-1}$ is a characteristic polynomial of the matrix $U=M^{-1}M^*$.
\item The curve $C_c$ meets $z=0$ at $(1,0,0)^t$ and $(0,1,0)^t$, neither of which is an inflexion point.
\end{enumerate}
\end{lemma}
\begin{proof}
The proof is omitted as the result can be easily verified through a direct calculation.
\end{proof}
By excluding the line $z=0$ from the projective plane, we focus on the affine part of the curve, denoted by $C_c^\circ$. Since the intersection of $C_c$ and the line $z=0$ consists precisely of the two points $(1,0,0)^t$ and $(0,1,0)^t$, the set $C_c^\circ$ represents $C_c$ minus these two points. By setting $z=1$ in \eqref{eq:Cc}, we obtain the affine equation for $C_c^\circ$:
\begin{equation}\label{eq:affCc}
x^{\sqrt{q}} - cx(y^{\sqrt{q}} -1) - y = 0.
\end{equation}
According to Proposition \ref{degsqrtq}, the coefficient of the linear term in $x$ plays a crucial role in determining the number of roots. This motivates us to partition the field $\mathbb{F}_q$ based on the value of the norm of this coefficient. Specifically, we define the following sets:
\begin{equation*}
\mathcal{Y}_1 = \{ y \in \mathbb{F}_q \mid \mathrm{Nm}(c(y^{\sqrt{q}}-1)) = 1 \}, \quad \mathcal{Y}_0 = \mathbb{F}_q \setminus \mathcal{Y}_1.
\end{equation*}
This classification allows us to determine the number of solutions for each case.
\begin{lemma}\label{mathcalY} For the curve $C_c$, the following properties hold:
\begin{enumerate}[{\rm(i)}]
\item $|\mathcal{Y}_1| = \sqrt{q} + 1$ and $|\mathcal{Y}_0| =q-( \sqrt{q} + 1)$.
\item For each $y_0 \in \mathcal{Y}_0$, the line $y = y_0$ intersects $C_c^\circ$ at exactly one $\mathbb{F}_q$-point.
\item\label{sqline} For each $y_0 \in \mathcal{Y}_1$, the following are equivalent:
\begin{enumerate}[{\rm(a)}]
\item The line $y = y_0$ intersects the affine curve $C_c^\circ$ at an $\mathbb{F}_q$-point.
\item The line $y = y_0$ is a $\mathbb{F}_q$-rational $\sqrt{q}$-line of $C_c^\circ$; that is, $|(y = y_0) \cap C_c^\circ(\mathbb{F}_q)| = \sqrt{q}$.
\item $y_0 = 0$ or $y_0^{1-\sqrt{q}} = -c(y_0^{\sqrt{q}} - 1)$.
\item $y_0 = 0$ or $y_0^{\sqrt{q}-1} +c^{\sqrt{q}} y_0  - c^{\sqrt{q}}=0$.
\end{enumerate}
\item\label{ssqrtq} Let $s$ be the number of points in $\mathcal{Y}_1$ that satisfy the equivalent conditions in item \ref{sqline}. Then, the total number of $\mathbb{F}_q$-rational points on $C_c$ is given by
$$N_q(C_c) = (q - \sqrt{q} + 1) + s\sqrt{q}.$$
\end{enumerate}
\end{lemma}
\begin{proof}
(i) Since the map $y \mapsto c(y^{\sqrt{q}}-1)$ is a bijection on $\mathbb{F}_q$ and $|\ker(\mathrm{Nm})| = \sqrt{q} + 1$, it follows that $|\mathcal{Y}_1| = \sqrt{q} + 1$, and consequently $|\mathcal{Y}_0| = q - \sqrt{q} - 1$.
(ii) If $y_0 \in \mathcal{Y}_0$, then $\mathrm{Nm}(-c(y_0^{\sqrt{q}}-1)) = \mathrm{Nm}(c(y_0^{\sqrt{q}}-1)) \neq 1$. According to Proposition \ref{degsqrtq}, it follows that equation \eqref{eq:affCc} possesses exactly one solution as an equation in $x$. This implies that the line $y = y_0$ intersects $C_c^\circ$ at precisely one point.
(iii) For $y_0 \in \mathcal{Y}_1$, Proposition \ref{degsqrtq} implies that the equation has $\sqrt{q}$ distinct roots if and only if $y_0 = 0$ or $y_0^{1-\sqrt{q}} = -c(y_0^{\sqrt{q}} - 1)$. Equivalently, this condition can be written as $y_0 = 0$ or $y_0^{\sqrt{q}-1} = -c^{\sqrt{q}}(y_0 - 1)$. If neither condition is satisfied, the equation possesses no roots.
(iv) There are two $\mathbb{F}_q$-points on the line $z = 0$. The affine plane where $z \neq 0$ can be expressed as a disjoint union of $q$ lines of the form $y = y_0$ for $y_0 \in \mathbb{F}_q$. Among these, for the $q - \sqrt{q} - 1$ values of $y_0$ belonging to $\mathcal{Y}_0$, each corresponding line contains exactly one $\mathbb{F}_q$-point. Combining these with the two points at infinity, we obtain $q - \sqrt{q} + 1$ points. Furthermore, for $y_0 \in \mathcal{Y}_1$, we have defined $s$ as the number of such points satisfying the conditions in item (3). Since each of these $s$ lines contributes $\sqrt{q}$ points, an additional $s\sqrt{q}$ points are added to the total.
\end{proof}
\begin{corollary} 
Every Hermitian-relative curve possessing at least one $\mathbb{F}_q$-point that is not an inflexion has at least $q - \sqrt{q} + 1$ $\mathbb{F}_q$-rational points. 
\end{corollary}
\begin{proof}
A Hermitian-relative curve $C$ possessing at least one $\mathbb{F}_q$-point that is not an inflexion is projectively equivalent to the curve $C_{a,b}$ defined by \eqref{eqCab}. In the case where $b=0$, it follows from Proposition \ref{N_qCa0} that $N_q(C_{a,0}) \ge q - \sqrt{q} + 1$. When $b \neq 0$, the curve is projectively equivalent to $C_c$ defined by \eqref{eq:Cc} according to Proposition \ref{prop:Cc}. Consequently, Lemma \ref{mathcalY} \ref{ssqrtq} implies that $N_q(C_c) \ge q - \sqrt{q} + 1$, which completes the proof.
\end{proof}
According to Lemma \ref{mathcalY} \ref{ssqrtq}, the question of determining $N_q(C_c)$ reduces to finding the value of $s$ as defined therein.
\begin{lemma} \label{eigenvec}
For an eigenvalue $\lambda_0$ of $U$, $(\frac{1}{c \lambda_0 }, 1-c^{-\sqrt{q}} \lambda_0 , 1)^t = (\frac{1}{c \lambda_0 }, \frac{-1+c \lambda_0}{c\lambda_0^2} , 1)^t $ is  the corresponding eigenvector.
\end{lemma}
\begin{proof} Let $(x_0, y_0, 1)^t$ be an eigenvector of $U$ corresponding to the eigenvalue $\lambda_0$. Then we have the following matrix equation:
\begin{equation}\label{matrixeq}
\begin{pmatrix} 0 & 0 & c^{-1} \\ -1 & 0 & 1 \\ 0 & -c^{\sqrt{q}} & c^{\sqrt{q}} \end{pmatrix} \begin{pmatrix} x_0 \\ y_0 \\ 1 \end{pmatrix} = \lambda_0 \begin{pmatrix} x_0 \\ y_0 \\ 1 \end{pmatrix}.
\end{equation}
This leads to the following system of equations:
\begin{equation}\label{systeq}
\begin{cases}c^{-1} = \lambda_0 x_0 \\ -x_0 + 1 = \lambda_0 y_0 \\ -c^{\sqrt{q}} y_0 + c^{\sqrt{q}} = \lambda_0
\end{cases}\end{equation}
From these relations, we can express $x_0$ and $y_0$ in terms of $\lambda_0$ as follows:
\begin{equation}
x_0 = \frac{1}{c\lambda_0} \quad \text{and} \quad y_0 = 1 - c^{-\sqrt{q}} \lambda_0 = \frac{c\lambda_0 - 1}{c\lambda_0^2}.
\end{equation}
\end{proof}
The following lemma shows the important role of the eigenvalues of $U$.
\begin{lemma} \label{eigenval}
The following properties hold:
\begin{enumerate}[{\rm (i)}]
\item If $(x_0, y_0, 1)^t$ is an inflexion of the curve, then $\frac{1}{cx_0}$ is a root of $P_U(\lambda)$.
\item If the line $y=y_0$ is a $\sqrt{q}$-line of the affine part $C_c^\circ$, then $c^{\sqrt{q}}(1-y_0)$ is a root of $P_U(\lambda)$. Furthermore, this root is a $(\sqrt{q}+1)$-th root of unity.
\end{enumerate}
\end{lemma}
\begin{proof}
(i) If $(x_0 , y_0 , 1)^t $ is an inflexion, then it is an eigenvector of $U$ by Proposition \ref{eigenvector}. If $\lambda_0$ is the corresponding eigenvalue to it, we have \eqref{matrixeq}. From \eqref{systeq} we have $\lambda_0 = \frac{1}{cx_0}$. 

(ii) If the line $y=y_0$ is a $\sqrt{q}$-line of $C_c^\circ$, then Lemma \ref{mathcalY} implies that $y_0 \in \mathcal{Y}_1$ and that either $y_0=0$ or $y_0^{1-\sqrt{q}}=-c(y_0^{\sqrt{q}}-1)$ holds.
In the case $y_0 = 0$, the condition $y_0 \in \mathcal{Y}_1$ implies $\mathrm{Nm}(c) = \mathrm{Nm}(-c) = 1$. Consequently, we have:
$$P_U(c^{\sqrt{q}}(1-y_0)) = P_U(c^{\sqrt{q}}) = (c^{\sqrt{q}})^3 - c^{\sqrt{q}}(c^{\sqrt{q}})^2 + c^{\sqrt{q}}(c^{\sqrt{q}}) - c^{\sqrt{q}-1} = 0,$$
since $c^{\sqrt{q}+1} = \mathrm{Nm}(c) = 1$ implies $c^{\sqrt{q}} = c^{-1}$.
In the case where $y_0^{1-\sqrt{q}} = -c(y_0^{\sqrt{q}} - 1)$, raising both sides to the $\sqrt{q}$-th power yields $y_0^{\sqrt{q}-1} = -c^{\sqrt{q}}(y_0 - 1)$. 
Hence we obtain:
$$\begin{aligned}
P_U(c^{\sqrt{q} }(1-y_0 ))  = P_U(y_0^{\sqrt{q}-1}) &= (y_0^{\sqrt{q}-1})^3 - c^{\sqrt{q}} (y_0^{\sqrt{q}-1})^2 + c^{\sqrt{q}} (y_0^{\sqrt{q}-1}) - c^{\sqrt{q}-1} \\
&= (y_0^{\sqrt{q}-1} - c^{\sqrt{q}}) (y_0^{\sqrt{q}-1})^2 + c^{\sqrt{q}} (y_0^{\sqrt{q}-1}) - c^{\sqrt{q}-1}\\
&= (-c^{\sqrt{q}}y_0) (y_0^{\sqrt{q}-1})^2 + c^{\sqrt{q}} (y_0^{\sqrt{q}-1}) - c^{\sqrt{q}-1} \\
&= -c^{\sqrt{q}} y_0^{2\sqrt{q}-1} + c^{\sqrt{q}} y_0^{\sqrt{q}-1} - c^{\sqrt{q}-1} \\
&= c^{\sqrt{q}} y_0^{\sqrt{q}-1} (-y_0^{\sqrt{q}} + 1) - c^{\sqrt{q}-1} \\
&= c^{\sqrt{q}-1}  (y_0^{\sqrt{q}-1}) (-c(y_0^{\sqrt{q}} - 1)) - c^{\sqrt{q}-1} \\
&= c^{\sqrt{q}-1} (y_0^{\sqrt{q}-1} \cdot y_0^{1-\sqrt{q}}) - c^{\sqrt{q}-1} = 0.
\end{aligned}$$
Thus  $c^{\sqrt{q} }(1-y_0 )$ is a root of $P_U (\lambda)$.
Finally, since $y_0 \in \mathcal{Y}_1$, we have $\mathrm{Nm}(c(y_0^{\sqrt{q}}-1)) = 1$. This implies that the element $c^{\sqrt{q}}(1-y_0)$ is also a $(\sqrt{q}+1)$-th root of unity, as its norm satisfies:
$$\mathrm{Nm}(c^{\sqrt{q}}(1-y_0)) =\mathrm{Nm}(c(y_0^{\sqrt{q}}-1)) = 1.$$
\end{proof}
Lemmas \ref{eigenvec} and  \ref{eigenval} suggest the importance of the eigenvalues of $U = M^{-1}M^*$.
Recall that two matrices $A$ and $B$ are said to be similar if there exists a nonsingular matrix $T$ such that $B = T^{-1}AT$. Note that similar matrices share the same characteristic polynomial.
\begin{proposition} For $U=M^{-1}M^*$, we have the following:
\begin{enumerate}[{\rm(i)}]\label{U-1*}
\item $U^* $ and $U^{-1}$ are similar to each other. So we have $P_{U^*} (\lambda)= 
P_{U^{-1}}(\lambda)$. 
\item If $\lambda_0$ is an eigenvalue of $U$, then $\lambda_0^{\sqrt{q}}=\lambda_1^{-1}$ for some eigenvalue $\lambda_1$ of $U$. Note that $\lambda_0$ need not be equal to $\lambda_1$.
\end{enumerate}
\end{proposition}
\begin{proof}
(i) Since$$M^{-1}U^*M = M^{-1} (M^{-1}M^*)^* M = M^{-1}M (M^{-1})^* M = (M^{-1})^* M = U^{-1},$$it follows that $U^*$ and $U^{-1}$ are similar. Consequently, their characteristic polynomials coincide, i.e., $P_{U^*}(\lambda) = P_{U^{-1}}(\lambda)$.

(ii) Note that the eigenvalues of $U^*$ are the $\sqrt{q}$-th powers of those of $U$, while the eigenvalues of $U^{-1}$ are the reciprocals of the eigenvalues of $U$. 
Let $\lambda_1, \lambda_2, \lambda_3 \in \overline{\mathbb{F}}_q$ be the eigenvalues of $U$, counted with multiplicity. Then the eigenvalues of $U^*$ and $U^{-1}$ are given by $\lambda_1^{\sqrt{q}}, \lambda_2^{\sqrt{q}}, \lambda_3^{\sqrt{q}}$ and $\lambda_1^{-1}, \lambda_2^{-1}, \lambda_3^{-1}$, respectively. These two sequences of eigenvalues are identical up to reordering by   (i), which leads to (ii).
\end{proof}
Using this lemma, we classify the characteristic polynomials $P_{U_c}$ for $c \in \mathbb{F}_q^*$, where $U_c = M_c^{-1} M_c^*$. Recall that $M_c = \begin{pmatrix} 0 & 0 & 1 \\ -c & 0 & 0 \\ c & -1 & 0 \end{pmatrix}$ is the matrix associated with the curve $C_c$ defined by equation \eqref{eq:Cc}.
\begin{theorem}\label{5polys}
For each $c \in \mathbb{F}_q^*$, the polynomial $P_{U_c}(\lambda)$ either splits completely into linear factors over $\mathbb{F}_q$ or is irreducible over $\mathbb{F}_q$. More precisely, exactly one of the following holds:
\begin{enumerate}[\rm(i)]
\item $P_{U_c}(\lambda)$ has a triple root. This occurs only when $\lambda=1$,  $c=3$ and $p \neq 3$.
\item $P_{U_c}(\lambda)$ has a double root and a simple root, both of which are $(\sqrt{q}+1)$-th roots of unity.
\item $P_{U_c}(\lambda)$ has three distinct simple roots, all of which are $(\sqrt{q}+1)$-th roots of unity.
\item $P_{U_c}(\lambda)$ has three distinct simple roots, only one of which is a $(\sqrt{q}+1)$-th root of unity, while the other two, say $\beta$ and $\gamma$, satisfy $\beta^{\sqrt{q}} = \gamma^{-1}$.
\item $P_{U_c}(\lambda)$ is irreducible over $\mathbb{F}_q$.
\end{enumerate}
\end{theorem}
\begin{proof}
As in the lemma above, let $\lambda_1, \lambda_2, \lambda_3 \in \overline{\mathbb{F}}_q$ be  the eigenvalues of $U_c$, counted with multiplicity. We first classify the cases based on whether $P_{U_c}(\lambda)$ has a triple root, a double root, or three distinct roots. 

{\bf Case 1:} $P_{U_c}(\lambda)$ has a triple root. In this case, we have $\lambda_1 = \lambda_2 = \lambda_3$. From the relation $\{\lambda_i^{\sqrt{q}}\} = \{\lambda_i^{-1}\}$, it follows that $\lambda_1^{\sqrt{q}+1} = 1$, which implies $\lambda_1 \in \mathbb{F}_q$. The characteristic polynomial is then expressed as:
$$P_{U_c}(\lambda) = (\lambda - \lambda_1)^3 = \lambda^3 - 3\lambda_1\lambda^2 + 3\lambda_1^2\lambda - \lambda_1^3.$$
By comparing these coefficients with $P_{U_c}(\lambda) = \lambda^3 - c^{\sqrt{q}} \lambda^2 +  c^{\sqrt{q}} \lambda - c^{\sqrt{q}-1}$ and recalling $c\ne 0$ and $\lambda_1 \ne 0$, we know that $\lambda_1=1$, $c=3$ and $p\ne 3$.

{\bf Case 2:}  Suppose $P_{U_c}(\lambda)$ has a double root and a simple root, and let $\lambda_1 = \lambda_2 \ne \lambda_3$. Since the operation of taking $\sqrt{q}$-th powers and reciprocals preserves the multiplicities of the eigenvalues, it follows directly that $\lambda_1^{\sqrt{q}} = \lambda_1^{-1}$ and $\lambda_3^{\sqrt{q}} = \lambda_3^{-1}$. Thus, $\lambda_1^{\sqrt{q}+1} = 1 = \lambda_3^{\sqrt{q}+1}$.

{\bf Case 3:}  $P_{U_c}(\lambda)$ has three distinct roots. Let $\{\lambda_1, \lambda_2, \lambda_3\}$ be the set of eigenvalues of $U_c$ over $\overline{\mathbb{F}}_q$. Then the sets $\{\lambda_1^{\sqrt{q}}, \lambda_2^{\sqrt{q}}, \lambda_3^{\sqrt{q}}\}$ and $\{\lambda_1^{-1}, \lambda_2^{-1}, \lambda_3^{-1}\}$ correspond to the eigenvalues of $U_c^*$ and $U_c^{-1}$, respectively. Since these sets must coincide, for each $i \in \{1, 2, 3\}$, we can define a permutation $\tau$ on $\{1, 2, 3\}$ such that $\lambda_i^{\sqrt{q}} = \lambda_{\tau(i)}^{-1}$. Without loss of generality, we may assume that $\tau$ is the identity $(1)$, the transposition $(12)$, or the $3$-cycle $(123)$. 

{\bf Subcase 3.1.} $\tau = (1)$: In this case, $\lambda_i^{\sqrt{q}} = \lambda_i^{-1}$ for all $i=1, 2, 3$, which implies $\lambda_i^{\sqrt{q}+1} = 1$. Each $\lambda_i$ is a $(\sqrt{q}+1)$-th root of unity and thus belongs to $\mathbb{F}_q$. This corresponds to case (iii) in the Theorem.

{\bf Subcase 3.2.} $\tau = (12)$: The relations $\lambda_1^{\sqrt{q}} = \lambda_2^{-1}$, $\lambda_2^{\sqrt{q}} = \lambda_1^{-1}$, and $\lambda_3^{\sqrt{q}} = \lambda_3^{-1}$ imply that $\lambda_3$ is a $(\sqrt{q}+1)$-th root of unity. Furthermore, the first two relations yield:$$\lambda_1^q = (\lambda_1^{\sqrt{q}})^{\sqrt{q}} = (\lambda_2^{-1})^{\sqrt{q}} = (\lambda_2^{\sqrt{q}})^{-1} = (\lambda_1^{-1})^{-1} = \lambda_1.$$Thus, $\lambda_1 \in \mathbb{F}_q$, and similarly, $\lambda_2 \in \mathbb{F}_q$. This aligns with case (iv).

{\bf Subcase 3.3.}  $\tau = (123)$: The relations are given by $\lambda_1^{\sqrt{q}} = \lambda_2^{-1}$, $\lambda_2^{\sqrt{q}} = \lambda_3^{-1}$, and $\lambda_3^{\sqrt{q}} = \lambda_1^{-1}$. We then observe:
$$\lambda_1^q = (\lambda_1^{\sqrt{q}})^{\sqrt{q}} = (\lambda_2^{-1})^{\sqrt{q}} = (\lambda_2^{\sqrt{q}})^{-1} = (\lambda_3^{-1})^{-1} = \lambda_3.$$ By a similar computation, we find $\lambda_2^q = \lambda_1$ and $\lambda_3^q = \lambda_2$. Since the $q$-th power of each $\lambda_i$ is mapped to $\lambda_{\tau(i)}\ne \lambda_i$, no $\lambda_i$ belongs to $\mathbb{F}_q$. Consequently, $P_{U_c}(\lambda)$ is irreducible over $\mathbb{F}_q$ since it is a cubic polynomial.
This corresponds to case (v).
\end{proof}
\begin{remark}
For any $c \in \mathbb{F}_q^*$, the cubic polynomial $P_{U_c}(\lambda)$ does not have any irreducible factor of degree 2 over $\mathbb{F}_q$. Furthermore, excluding the irreducible case, every $P_{U_c}(\lambda)$ possesses at least one $(\sqrt{q}+1)$-th root of unity as a root.
\end{remark}
From here on, we use $\mu_{\sqrt{q}+1}$ to denote the set of $(\sqrt{q}+1)$-th roots of unity in $\mathbb{F}_q$; that is, $\mu_{\sqrt{q}+1} = \mathrm{Ker}(\mathrm{Nm}) = \{ \zeta \in \mathbb{F}_q \mid \zeta^{\sqrt{q}+1} = 1 \}$.
 
\begin{lemma}
If $\lambda_0 \in \mathbb{F}_q^*$ is a multiple root of $P_{U_c}(\lambda)$, then 
$\lambda_0 \in \mu_{\sqrt{q}+1}$ and  the eigenvector corresponding to it represents a point on the curve $C_c$, which is an inflexion.
\end{lemma}
\begin{proof}
By Case 1 and 2 in the proof of Theorem \ref{5polys}, we have $\lambda_0 \in \mu_{\sqrt{q}+1}$. The eigenvector corresponding $\lambda_0$ is $(\frac{1}{c \lambda_0 }, 1-c^{-\sqrt{q}} \lambda_0 , 1)^t$ by Lemma \ref{eigenvec}. Substituting this vector in the equation of \eqref{eq:Cc}, we have
$$
\begin{aligned}
&\left( \frac{1}{c \lambda_0} \right)^{\sqrt{q}} - c \left( \frac{1}{c \lambda_0} \right) ( (1-c^{-\sqrt{q}}\lambda_0 )^{\sqrt{q}} -1) -   (1 -c^{-\sqrt{q}}\lambda_0 )\\ 
&=\left( \frac{\lambda_0}{c^{\sqrt{q}} } \right)  -   \left( \frac{1}{  \lambda_0} \right)( -c^{-1}\lambda_0 ^{\sqrt{q}}  ) -   (1 -c^{-\sqrt{q}}\lambda_0 )\\ 
&=\left( \frac{1}{c^{\sqrt{q}} \lambda_0^2 } \right) \left( 2 \lambda_0^3 
 - c^{\sqrt{q}} \lambda_0^2 + c^{\sqrt{q}-1}  \right) \\ 
&=\left( \frac{1}{c^{\sqrt{q}} \lambda_0^2 } \right) 
\left(  \lambda_0 P'_{U_c}(\lambda_0) -    P_{U_c}(\lambda_0) \right) =0,
\end{aligned}
$$
where $ P'_{U_c}(\lambda)$ denotes the derivative of $P_{U_c}(\lambda)$.

\end{proof}
\begin{lemma} \label{musq}
If $\lambda_0 \in \mathbb{F}_q^*$ is a root of $P_{U_c}(\lambda)$ and 
$\lambda_0 \in\mu_{\sqrt{q}+1}$  , then the corresponding line $y=1-c^{-\sqrt{q}} \lambda_0 \ $ is a $\sqrt{q}$-line of $C_c^\circ$.
\end{lemma}
\begin{proof}
To show that the equation has $\sqrt{q}$ distinct roots, we apply Lemma \ref{mathcalY}. Let $y_0 = 1 - c^{-\sqrt{q}}\lambda_0$. First, we verify the norm condition of the coefficient of linear term:
$$\mathrm{Nm}(c((1-c^{-\sqrt{q}}\lambda_0)^{\sqrt{q}}-1))=\mathrm{Nm}(c( -c^{-1}\lambda_0^{\sqrt{q}} ) )= \mathrm{Nm}( -\lambda_0^{\sqrt{q}} )= 1,$$
where we use the relation $\lambda_0^{\sqrt{q}+1}=1$. According to Lemma 17 (iii), it suffices to show that either $y_0 = 0$ or $y_0^{\sqrt{q}-1} + c^{\sqrt{q}}y_0 - c^{\sqrt{q}} = 0$. Note that $y_0 = 0$ if and only if $\lambda_0 = c^{\sqrt{q}}$. In the case where $y_0 \neq 0$ (i.e., $\lambda_0 \neq c^{\sqrt{q}}$), we observe the following:$$\begin{aligned}
y_0^{\sqrt{q}-1} +c^{\sqrt{q}} y_0  - c^{\sqrt{q}}
&= \frac{1}{y_0} \left( y_0^{\sqrt{q}}+c^{\sqrt{q}} y_0^2  - c^{\sqrt{q}}y_0 \right)\\
&=\frac{1}{y_0}\left( (1-c^{-\sqrt{q}} \lambda_0)^{\sqrt{q}}+c^{\sqrt{q}} (1-c^{-\sqrt{q}} \lambda_0)^2  - c^{\sqrt{q}}(1-c^{-\sqrt{q}} \lambda_0) \right)\\
&=\frac{1}{y_0}\left(  1-c^{-1} \lambda_0^{\sqrt{q}}-\lambda_0 +c^{-\sqrt{q}}  \lambda_0^2   \right)\\
&=\frac{1}{y_0} \cdot c^{-\sqrt{q}} \lambda_0^{-1} \left(  c^{\sqrt{q}} \lambda_0 -c^{\sqrt{q}-1} \lambda_0^{\sqrt{q}+1}-c^{\sqrt{q}}\lambda_0^2 +   \lambda_0^3   \right)\\
&=\frac{1}{y_0} \cdot c^{-\sqrt{q}} \lambda_0^{-1} \left(   \lambda_0^3 -c^{\sqrt{q}}\lambda_0^2   +c^{\sqrt{q}} \lambda_0 -c^{\sqrt{q}-1}     \right)\\
&=\frac{1}{y_0} \cdot c^{-\sqrt{q}} \lambda_0^{-1} P_{U_c}(\lambda_0) = 0.
\end{aligned}$$
 Therefore, the condition in Lemma \ref{mathcalY} is satisfied, confirming that the equation possesses $\sqrt{q}$ distinct roots.
\end{proof}

\begin{lemma}
If $P_{U_c}(\lambda)$ has three distinct roots $\lambda_i$, $i=1,2,3$,  which belong to
$\mu_{\sqrt{q}+1}$, then the three corresponding lines $y=1-c^{-\sqrt{q}} \lambda_i $ are all $\sqrt{q}$-lines of $C_c^\circ$ and  $C_c$ has no inflexion.
\end{lemma}
\begin{proof}
By Lemma \ref{musq}, the three corresponding lines are $\sqrt{q}$-lines of $C_c^\circ$, and hence $N_q(C_c) = (q - \sqrt{q} + 1) + 3\sqrt{q}=q +2 \sqrt{q} + 1 $. By Lemma \ref{mirror-inter} (i), for each $P\in C_c(\mathbb{F}_q)$, we have $i(C_{M_c}.C_{M_c^*};P)=1$, hence $P$ cannot be an inflexion by Lemma \ref{mirror-inter} (ii). Here, $C_{M_c}$ denotes the curve corresponding to the matrix $M_c$, which is identified with $C_c$.
\end{proof}
\begin{lemma}
Suppose that $P_{U_c}(\lambda)$ has three distinct roots, exactly one of which belongs to $\mu_{\sqrt{q}+1}$. Then the two eigenvectors corresponding to the other two eigenvalues are inflexions of $C_c$. 
\end{lemma}\begin{proof}
Let $\lambda_1, \lambda_2, \lambda_3$ be the three distinct eigenvalues, and assume $\lambda_1 \in \mu_{\sqrt{q}+1}$. According to Theorem \ref{5polys}, this implies $\lambda_2^{\sqrt{q}} = \lambda_3^{-1}$ (and similarly $\lambda_3^{\sqrt{q}} = \lambda_2^{-1}$). From the relation between roots and coefficients of $P_{U_c}(\lambda) = (\lambda-\lambda_1)(\lambda-\lambda_2)(\lambda-\lambda_3)$, we obtain the following system of equations:
\begin{equation}\begin{cases}
\lambda_1 + \lambda_2 + \lambda_3 = c^{\sqrt{q}}, \\ \lambda_1(\lambda_2 + \lambda_3) + \lambda_2\lambda_3 = c^{\sqrt{q}}, \\ \lambda_1\lambda_2\lambda_3 = c^{\sqrt{q}-1}.
\end{cases}\end{equation}
By Lemma \ref{eigenvec}, the eigenvector corresponding to $\lambda_i$ is given by $v_i = ( \frac{1}{c\lambda_i}, 1-c^{-\sqrt{q}}\lambda_i, 1 )^t$. Since the proofs for $i=2$ and $i=3$ are analogous, we only provide the proof for $i=2$. Substituting $v_2$ into the curve equation \eqref{eq:Cc}, we have:
$$\begin{aligned}
&(c \lambda_2 )^{-\sqrt{q}} - \lambda_2^{-1} (1-c^{-1} \lambda_2^{\sqrt{q}}) + \lambda_2^{-1} - (1- c^{-\sqrt{q}}\lambda_2) \\
&= c^{-\sqrt{q}}\lambda_2^{-\sqrt{q}} + c^{-1} \lambda_2^{\sqrt{q}-1} - 1 + c^{-\sqrt{q}}\lambda_2 \\
&= c^{-\sqrt{q}}\lambda_2^{\sqrt{q}-1} \left( \lambda_2^{-2\sqrt{q}+1} + c^{\sqrt{q}-1} - c^{\sqrt{q}}\lambda_2^{-\sqrt{q}+1} + \lambda_2^{-\sqrt{q}+2} \right)\\
&= c^{-\sqrt{q}}\lambda_2^{\sqrt{q}-1} \left( \lambda_2 \lambda_3^2 + c^{\sqrt{q}-1} - c^{\sqrt{q}}\lambda_2 \lambda_3 + \lambda_2^2 \lambda_3 \right)\\
&= c^{-\sqrt{q}}\lambda_2^{\sqrt{q}-1} \left( \lambda_2 \lambda_3 (\lambda_2 + \lambda_3 - c^{\sqrt{q}}) + c^{\sqrt{q}-1} \right)\\
&= c^{-\sqrt{q}}\lambda_2^{\sqrt{q}-1} \left( -\lambda_1 \lambda_2 \lambda_3 + c^{\sqrt{q}-1} \right) = 0.
\end{aligned}$$
The substitution $\lambda_2^{-\sqrt{q}} = \lambda_3$ and the relation $\lambda_1 + \lambda_2 + \lambda_3 = c^{\sqrt{q}}$ were used in the final steps. Thus, $v_2$ is an inflexion point of $C_c$.
\end{proof}

\noindent Up to this point, we have used $P_{U_c}(\lambda)$ to denote the characteristic polynomial; however, hereafter we shall adopt the notation $f_c(\lambda)= \lambda^3 -c^{\sqrt{q}}\lambda^2 +c^{\sqrt{q}}\lambda -c^{\sqrt{q}-1} $ to emphasize our focus on the parameter $c$.

\begin{lemma}\label{deg2=1}
Let $\alpha, \beta, \gamma \in \mu_{\sqrt{q}+1}$. If $\alpha + \beta + \gamma = c^{\sqrt{q}} = \alpha \beta + \beta \gamma + \alpha \gamma$ for some $c \in \mathbb{F}_q^*$, then it necessarily follows that $\alpha \beta \gamma = c^{\sqrt{q}-1}$. Consequently, such $\alpha, \beta$, and $\gamma$ are precisely the roots of the polynomial $f_c$.
\end{lemma}
\begin{proof}
Let $\alpha\beta\gamma =d$. We have 
$$ \frac{1}{\alpha} + \frac{1}{\beta} + \frac{1}{\gamma} = \alpha^{\sqrt{q}} +   \beta^{\sqrt{q}} +  \gamma^{\sqrt{q}} = (\alpha+\beta+\gamma)^{\sqrt{q}} =(c^{\sqrt{q}})^{\sqrt{q}} =c.$$ 
On the other hand, we have 
$$ \frac{c^{\sqrt{q}}}{d} = \frac{\alpha\beta +\beta\gamma+\alpha\gamma}{\alpha\beta\gamma} =  \frac{1}{\alpha} + \frac{1}{\beta} + \frac{1}{\gamma} =c.$$
Thus $\alpha\beta\gamma =d= \frac{c^{\sqrt{q}}}{c} =c^{\sqrt{q}-1}$.

\end{proof}
Recall that $p$ denotes the characteristic of $\mathbb{F}_q$, i.e., $q$ is an even power of $p$.
\begin{lemma} \label{doublealpha} Let $p\ne 3$. For any  $\alpha \in \mu_{\sqrt{q}+1}$, 
there exists a unique element $c \in \mathbb{F}_q^*$ such that $\alpha$ is a double root of the polynomial $f_c$.
Moreover, the remaining root of $f_c(\lambda)$ is given by $\frac{\alpha(2-\alpha)}{2\alpha-1}$, which also belongs to $\mu_{\sqrt{q}+1}$.
\end{lemma}
\begin{proof}
Note that $2^{\sqrt{q}+1}=2^2 \ne 1$ since $p>3$, and hence $2, 2^{-1} \not\in \mu_{\sqrt{q}+1}$. Thus $\frac{\alpha(2-\alpha)}{2\alpha-1}$ is nonzero element in $\mathbb{F}_q$.
First, we observe that
$$\left( \frac{\alpha(2-\alpha)}{2\alpha-1} \right)^{\sqrt{q}} = \frac{\alpha^{\sqrt{q}}(2-\alpha^{\sqrt{q}})}{2\alpha^{\sqrt{q}}-1} = \frac{\alpha^{-1}(2-\alpha^{-1})}{2\alpha^{-1}-1} = \frac{2\alpha-1}{\alpha(2-\alpha)} = \left(\frac{\alpha(2-\alpha)}{2\alpha-1} \right)^{-1},$$
which implies that $\frac{\alpha(2-\alpha)}{2\alpha-1} \in \mu_{\sqrt{q}+1}$. Furthermore, since
$$2\alpha + \frac{\alpha(2-\alpha)}{2\alpha-1} = \frac{3\alpha^2}{2\alpha-1} \quad \text{and} \quad \alpha^2 + 2\alpha \left( \frac{\alpha(2-\alpha)}{2\alpha-1} \right) = \frac{3\alpha^2}{2\alpha-1},$$
the conditions of Lemma \ref{deg2=1} are satisfied. Consequently, the polynomial $f_c$ with $c = \left( \frac{3\alpha^2}{2\alpha-1} \right)^{\sqrt{q}}$ (or $c^{\sqrt{q}} = \frac{3\alpha^2}{2\alpha-1}$) possesses $\alpha$ as a double root and $\frac{\alpha(2-\alpha)}{2\alpha-1}$ as the third root.

Let $f_{c'}$ be a polynomial with $c' \in \mathbb{F}_q$ such that $f_{c'}$ has $\alpha$ as a double root, and let $\beta$ denote the remaining root. Then, the relations between the roots and coefficients imply:
$$2\alpha + \beta = c'^{\sqrt{q}} = \alpha^2 + 2\alpha\beta.
$$
From this equality, we obtain $\beta(2\alpha - 1) = 2\alpha - \alpha^2$, which yields:
$$\beta = \frac{\alpha(2-\alpha)}{2\alpha-1}.$$
This value of $\beta$ is identical to the one derived previously. Consequently, it follows that $c' = c$, establishing the uniqueness of the polynomial $f_c$ that possesses a given $\alpha \in \mu_{\sqrt{q}+1}$ as a double root.
\end{proof}
%
\begin{lemma}\label{p3doubleroot}
When $p=3$, if the polynomial $f_c$ (with $c \neq 0$) possesses a double root $\alpha$, then it must be that 
$\alpha = 2$. Specifically, for $\alpha=2 \in \mu_{\sqrt{q}+1}$ and any $\beta \in \mu_{\sqrt{q}+1} \setminus \{2\}$, the 
polynomial $f_c$ with $c = (1+\beta)^{\sqrt{q}}$ has $2$ as a double root and $\beta$ as the remaining simple root. Note that 
the case $\beta=2$ is excluded because it leads to $c=0$, which contradicts our assumption that $c \in \mathbb{F}_q^*$.
\end{lemma}
\begin{proof}
If $\alpha$ is a double root of some $f_c$ with $c\neq 0$, then $0=f_c'(\alpha) = -2c^{\sqrt{q}} \alpha + c^{\sqrt{q}}$, which implies that
$\alpha =2$.
Note that $2 \in \mu_{\sqrt{q}+1}$.

Let $f_c$ have $2$ as a double root and $\beta$ as the third root. Then:
$$f_c(\lambda) = (\lambda - 2)^2 (\lambda - \beta) = (\lambda + 1)^2 (\lambda - \beta) = \lambda^3 - (1+\beta)\lambda^2 + (1+\beta)\lambda - \beta.$$
Comparing coefficients yields $c^{\sqrt{q}} = 1+\beta$ and $\beta=c^{\sqrt{q}-1}$. Since $c \in \mathbb{F}_q^*$, it follows that $\beta \in \mu_{\sqrt{q}+1}$.
Indeed, one can verify that $2, 2$, and $\beta$ are the roots of $f_c$ with $c = (1+\beta)^{\sqrt{q}}=1+\beta^{\sqrt{q}}$.
\end{proof}
\begin{remark}
The polynomial $f_c$
has a triple root if and only if 
$\alpha =\frac{\alpha(2-\alpha)}{2\alpha-1}$, which is equivalent to $\alpha =1$ since $\alpha\ne 0$. In fact, when $c=3$, we have
$f_c(\lambda)= \lambda^3 -c^{\sqrt{q}} \lambda^2 + c^{\sqrt{q}} \lambda -c^{\sqrt{q}-1} =  \lambda^3 -3 \lambda^2 + 3 \lambda -1 =(\lambda-1)^3 $.
\end{remark}
Recall that $\omega$ is a primitive element of $\mathbb{F}_q$. We obtain the following result:
\begin{lemma}\label{prodsum1}
The solutions to the system of equations
$$\begin{cases} \alpha_1\alpha_2 = 1 \\ \alpha_1 + \alpha_2 = 1 \end{cases}$$
over $\mathbb{F}_q$ are given as follows:
\[
\{
\alpha_1 , \alpha_2
\}
=
\begin{cases}
\{ \omega^{\frac{q-1}{3}}, \omega^{\frac{2(q-1)}{3}} \} & p=2 \\
\{ -1, -1 \} & p=3 \\
\{ \omega^{\frac{q-1}{6}}, \omega^{\frac{5(q-1)}{6}} \} & p \neq 2, 3 .
\end{cases}
\]
\end{lemma}
\begin{proof}  Obvious.
\end{proof}
\begin{lemma}\label{sum1prod1}
Let $\alpha_1$ and $\alpha_2$ be distinct elements in $\mathbb{F}_q^*$ satisfying
$\{\alpha_1^{\sqrt{q}}, \alpha_2^{\sqrt{q}}\} = \{\alpha_1^{-1}, \alpha_2^{-1}\}$ and $\alpha_1 + \alpha_2 = 1$.
Then $\alpha_1 \alpha_2 = 1$. In the case where $\alpha_1^{\sqrt{q}}=\alpha_2^{-1}$ and $\alpha_2^{\sqrt{q}}=\alpha_1^{-1}$, we further have $\alpha_1, \alpha_2 \in \mathbb{F}_{\sqrt{q}}^*$.
\end{lemma}
\begin{proof}
Since $\alpha_1 + \alpha_2 = 1$, we have
\[
1 = (\alpha_1 + \alpha_2)^{\sqrt{q}} = \alpha_1^{\sqrt{q}} + \alpha_2^{\sqrt{q}} = \alpha_1^{-1} + \alpha_2^{-1} = \frac{\alpha_1 + \alpha_2}{\alpha_1 \alpha_2} = \frac{1}{\alpha_1 \alpha_2},
\]
and hence $\alpha_1 \alpha_2 = 1$. 

Finally, if $\alpha_1^{\sqrt{q}} = \alpha_2^{-1}$ and $\alpha_2^{\sqrt{q}} = \alpha_1^{-1}$, we have
\[
\alpha_1^{\sqrt{q}-1} = \alpha_1^{\sqrt{q}} \alpha_1^{-1} = \alpha_2^{-1} \alpha_1^{-1} = (\alpha_1 \alpha_2)^{-1} = 1.
\]
Similarly, it follows that $\alpha_2^{\sqrt{q}-1} = 1$, which implies $\alpha_1, \alpha_2 \in \mathbb{F}_{\sqrt{q}}^*$.
\end{proof}
\begin{lemma}\label{thirdone}
Let $\alpha, \beta \in \mu_{\sqrt{q}+1}$. Suppose that $\alpha$ and $\beta$ are roots of $f_c(\lambda)$ for some $c \in \mathbb{F}_q^*$.
\begin{enumerate}[\rm(i)]
\item If $\alpha + \beta \neq 1$ and $\alpha^2 + \alpha\beta + \beta^2 \neq 0$, then the third root $\gamma$ also belongs to $\mu_{\sqrt{q}+1}$. In this case, $\gamma$ and $c$ are explicitly given by
$$\gamma = \frac{\alpha + \beta - \alpha\beta}{\alpha + \beta - 1} \quad \text{and} \quad c = \left( \frac{\alpha^2 + \alpha\beta + \beta^2}{\alpha + \beta - 1} \right)^{\sqrt{q}}.$$
\item If $\alpha + \beta = 1$, then, 
for any $\gamma \in \mu_{\sqrt{q}+1} \setminus \{-1\}$, the elements $\alpha, \beta, \gamma$ are the roots of $f_c(\lambda)$ with $c = (1 + \gamma)^{\sqrt{q}} = 1 + \gamma^{\sqrt{q}}$.
\end{enumerate}
\end{lemma}
\begin{proof}
(i) Let $\gamma = \frac{\alpha + \beta - \alpha\beta}{\alpha + \beta - 1}$. Then we have:
$$\gamma^{\sqrt{q}} = \left( \frac{\alpha + \beta - \alpha\beta}{\alpha + \beta - 1} \right)^{\sqrt{q}} = \frac{\alpha^{-1} + \beta^{-1} - \alpha^{-1}\beta^{-1}}{\alpha^{-1} + \beta^{-1} - 1} = \frac{\alpha + \beta - 1}{\alpha + \beta - \alpha\beta} = \gamma^{-1},$$
which implies $\gamma^{\sqrt{q}+1} = 1$. Furthermore, a direct calculation shows that:
$$\alpha + \beta + \gamma = \alpha + \beta + \frac{\alpha + \beta - \alpha\beta}{\alpha + \beta - 1} = 
 \frac{\alpha^2 + \alpha\beta + \beta^2}{\alpha + \beta - 1},$$
and
$$
\alpha\beta + \beta\gamma + \alpha\gamma 
= \alpha\beta + (\alpha + \beta)\gamma = \alpha\beta + (\alpha + \beta) \frac{\alpha + \beta - \alpha\beta}{\alpha + \beta - 1} 
= \frac{\alpha^2 + \alpha\beta + \beta^2}{\alpha + \beta - 1}.
$$
By Lemma \ref{deg2=1}, it follows that $\alpha, \beta, \gamma$ are the roots of the polynomial $\lambda^3 - c^{\sqrt{q}}\lambda^2 + c^{\sqrt{q}}\lambda - c^{\sqrt{q}-1}$ where $c = \left( \frac{\alpha^2 + \alpha\beta + \beta^2}{\alpha + \beta - 1} \right)^{\sqrt{q}}$.

(ii) If we let $\alpha + \beta = 1$, then $\alpha\beta=1$ by Lemma \ref{sum1prod1}. 
Let $\gamma \in \mu_{\sqrt{q}+1} \setminus \{-1\}$ and set $c = (1 + \gamma)^{\sqrt{q}} = 1 + \gamma^{\sqrt{q}}$. Then we have:
$$\alpha + \beta + \gamma = 1 + \gamma = (1 + \gamma)^{q} = c^{\sqrt{q}},$$
and
$$\alpha\beta + \beta\gamma + \alpha\gamma = 1 + \gamma(\alpha + \beta) = 1 + \gamma = c^{\sqrt{q}}.$$
By Lemma \ref{deg2=1}, $\alpha, \beta$, and $\gamma$ are indeed the roots of the polynomial $f_c$ for $c = (1 + \gamma)^{\sqrt{q}}$.
\end{proof}
\begin{remark} \label{sum1a2abb20} If $\alpha + \beta = 1$, then $\alpha\beta = 1$ and 
$\alpha^2 + \alpha\beta + \beta^2=(\alpha + \beta)^2 - \alpha\beta = 1^2 - 1 = 0$.
\end{remark}
\begin{remark} In the context of Lemma \ref{thirdone} (2), the condition $\gamma \neq -1$ is assumed because $c = 0$ if $\gamma = -1$. Suppose $\alpha$ and $\beta$ are $(\sqrt{q}+1)$-th roots of unity satisfying $\alpha + \beta = 1$ and $\alpha\beta = 1$. In this case, $\alpha^2 + \alpha\beta + \beta^2 = (\alpha + \beta)^2 - \alpha\beta = 1 - 1 = 0$. For any $q$ odd, there exist exactly $\sqrt{q}$ distinct $(\sqrt{q}+1)$-th roots of unity $\gamma$ (excluding $-1$) such that $\alpha, \beta$, and $\gamma$ serve as the roots of the polynomial$$f_c(\lambda) = \lambda^3 - c^{\sqrt{q}}\lambda^2 + c^{\sqrt{q}}\lambda - c^{\sqrt{q}-1}.$$Note that $\gamma$ may coincide with $\alpha$ or $\beta$, which corresponds to the case where the polynomial $f_c(\lambda)$ possesses a multiple root.
\end{remark}

\begin{lemma}\label{betgamconj}
Let $\alpha\in\mu_{\sqrt{q}+1}$. Let $\beta$ and $\gamma$ be elements in $\mathbb{F}_q$ satisfying $\gamma = \beta^{-\sqrt{q}}$. If
$$\alpha + \beta + \gamma = c^{\sqrt{q}} = \alpha\beta + \beta\gamma + \gamma\alpha$$
for some $c \in \mathbb{F}_q$, then $\alpha\beta\gamma = c^{\sqrt{q}-1}$. In particular, $\alpha, \beta, \gamma$ are the roots of the polynomial
$$f_c(\lambda) = \lambda^3 - c^{\sqrt{q}}\lambda^2 + c^{\sqrt{q}}\lambda - c^{\sqrt{q}-1}.$$
\end{lemma}
\begin{proof}
Let $\alpha\beta\gamma = d$. Since $\alpha \in \mu_{\sqrt{q}+1}$, we have $\alpha^{-1} = \alpha^{\sqrt{q}}$. From the condition $\gamma = \beta^{-\sqrt{q}}$, it follows that $\gamma^{-1} = \beta^{\sqrt{q}}$ and $\beta^{-1} = \gamma^{\sqrt{q}}$. 
The remainder of the proof proceeds by an argument similar to that of Lemma \ref{deg2=1}.
\end{proof}
\begin{lemma}\label{betgamconjfc}
Let $\beta$ and $\gamma$ be elements in $\mathbb{F}_q$ such that $\beta^{\sqrt{q}} = \gamma^{-1}$ (or equivalently, $\beta \gamma^{\sqrt{q}} = 1$).
\begin{enumerate}[\rm (i)]
\item If $\beta + \gamma \neq 1$ and $\beta^2 + \beta\gamma + \gamma^2 \neq 0$, let
$$\alpha = \frac{\beta + \gamma - \beta\gamma}{\beta + \gamma - 1}.$$
Then $\alpha, \beta, \gamma$ are the roots of the polynomial
$$\lambda^3 - c^{\sqrt{q}}\lambda^2 + c^{\sqrt{q}}\lambda - c^{\sqrt{q}-1} \quad \text{for} \quad c = \left( \frac{\beta^2 + \beta\gamma + \gamma^2}{\beta + \gamma - 1} \right)^{\sqrt{q}}.$$
Moreover, $\alpha$ is a $(\sqrt{q}+1)$-th root of unity.
\item If $\beta + \gamma = 1$, then 
the roots of the polynomial
$$\lambda^3 - c^{\sqrt{q}}\lambda^2 + c^{\sqrt{q}}\lambda - c^{\sqrt{q}-1}$$
are $\{c^{\sqrt{q}-1}, \beta, \gamma\}$, where $c$ satisfies the equation $c^{\sqrt{q}} - c^{\sqrt{q}-1} - 1 = 0$. Indeed, this equation has exactly $\sqrt{q}$ distinct roots for such $c$.
\end{enumerate}
\end{lemma}
\begin{proof}
(i)  Assume $\beta + \gamma \neq 1$ and $\beta^2 + \beta\gamma + \gamma^2 \neq 0$. Let $\alpha$ be defined as:
$$\alpha = \frac{\beta + \gamma - \beta\gamma}{\beta + \gamma - 1}.$$
First, we verify that $\alpha$ is a $(\sqrt{q}+1)$-th root of unity. Taking the $\sqrt{q}$-th power of $\alpha$ and substituting the relations $\beta^{\sqrt{q}} = \gamma^{-1}$ and $\gamma^{\sqrt{q}} = \beta^{-1}$, we obtain:
$$\alpha^{\sqrt{q}} = \frac{\beta^{\sqrt{q}} + \gamma^{\sqrt{q}} - \beta^{\sqrt{q}}\gamma^{\sqrt{q}}}{\beta^{\sqrt{q}} + \gamma^{\sqrt{q}} - 1} = \frac{\gamma^{-1} + \beta^{-1} - \gamma^{-1}\beta^{-1}}{\gamma^{-1} + \beta^{-1} - 1} = \frac{\beta + \gamma - 1}{\beta + \gamma - \beta\gamma} = \alpha^{-1}.$$
Thus, $\alpha^{\sqrt{q}+1} = 1$, confirming $\alpha \in \mu_{\sqrt{q}+1}$. Next, we evaluate the elementary symmetric polynomials. For $c = \left( \frac{\beta^2 + \beta\gamma + \gamma^2}{\beta + \gamma - 1} \right)^{\sqrt{q}}$, we have:
$$\alpha + \beta + \gamma = \frac{\beta + \gamma - \beta\gamma}{\beta + \gamma - 1} + \beta + \gamma = \frac{\beta^2 + \beta\gamma + \gamma^2}{\beta + \gamma - 1} = c^{\sqrt{q}},$$
and
$$\alpha\beta + \beta\gamma + \gamma\alpha = \alpha(\beta + \gamma) + \beta\gamma = \frac{\beta + \gamma - \beta\gamma}{\beta + \gamma - 1}(\beta + \gamma) + \beta\gamma = \frac{\beta^2 + \beta\gamma + \gamma^2}{\beta + \gamma - 1} = c^{\sqrt{q}}.$$
Finally, let $d = \alpha\beta\gamma$. We observe that:$$\frac{c^{\sqrt{q}}}{d} = \frac{\alpha\beta + \beta\gamma + \gamma\alpha}{\alpha\beta\gamma} = \frac{1}{\gamma} + \frac{1}{\alpha} + \frac{1}{\beta} = \alpha^{\sqrt{q}} + \gamma^{\sqrt{q}} + \beta^{\sqrt{q}} = (\alpha + \gamma + \beta)^{\sqrt{q}} = (c^{\sqrt{q}})^{\sqrt{q}} = c.$$It follows that $\alpha\beta\gamma = d = c^{\sqrt{q}-1}$. By the relations of roots and coefficients, $\alpha, \beta, \gamma$ are indeed the roots of $\lambda^3 - c^{\sqrt{q}}\lambda^2 + c^{\sqrt{q}}\lambda - c^{\sqrt{q}-1}$.

(ii) If $\beta + \gamma = 1$, then $\beta\gamma =1$ and $\beta,\gamma \in \mathbb{F}_{\sqrt{q}}^*$ by Lemma \ref{sum1prod1}.
Let $c$ be an element in $\mathbb{F}_q$ satisfying the equation $c^{\sqrt{q}} - c^{\sqrt{q}-1} - 1 = 0$. 
Since $c^{\sqrt{q}-1} + (\beta + \gamma) = c^{\sqrt{q}-1} + 1 = c^{\sqrt{q}}$ and
$c^{\sqrt{q}-1}(\beta + \gamma) + \beta\gamma = c^{\sqrt{q}-1} + 1 = c^{\sqrt{q}}$,
three elements $c^{\sqrt{q}-1}, \beta, \gamma$ are the roots of the polynomial $\lambda^3 - c^{\sqrt{q}}\lambda^2 + c^{\sqrt{q}}\lambda - c^{\sqrt{q}-1}$. Finally, we note that the equation $c^{\sqrt{q}} - c^{\sqrt{q}-1} - 1 = 0$ can be rewritten as $-c^{\sqrt{q}}(-1 + \frac{1}{c} + \frac{1}{c^{\sqrt{q}}}) = 0$. Since $(-1)^{1-\sqrt{q}} = 1$ for any $q$, according to Proposition \ref{degsqrtq}, this equation possesses exactly $\sqrt{q}$ distinct roots.
\end{proof}
\begin{lemma}\label{3rdroot}
Let $\mu_3$ be the group of the third roots of unity,
which is a subgroup of $\mathbb{F}_q^{\ast}$.
\begin{enumerate}[\rm(i)]
\item If $\sqrt{q} \equiv 0 \pmod 3$, then $\mu_3 = \{1\}$.
\item If $\sqrt{q} \not\equiv 0 \pmod 3$, then
$|\mu_3| =3$, and
\begin{enumerate}[\rm(a)]
\item $\mu_3 \subset \mu_{\sqrt{q}+1}$ if and only if
$\sqrt{q} \equiv 2 \pmod 3$;
\item $\mu_3 \subset \mathbb{F}_{\sqrt{q}}^{\ast}$
if and only if
$\sqrt{q} \equiv 1 \pmod 3$.
\end{enumerate}
\end{enumerate}
\end{lemma}
\begin{proof}
The assertions (i) and the first part of (ii) are trivial.
The remaining parts come from
the representations of groups:
$\mathbb{F}_q^{\ast}= \langle \omega \rangle$,
$\mu_{\sqrt{q}+1} =\langle \omega^{\sqrt{q}-1}\rangle$,
$\mathbb{F}_{\sqrt{q}}^{\ast} = \langle \omega^{\sqrt{q}+1}\rangle$
and $\mu_3 =\langle \omega^{\frac{q-1}{3}}\rangle$.
\end{proof}

\begin{lemma}\label{mua+b=1}
Let $\alpha, \beta \in \mu_{\sqrt{q}+1}$ and $\alpha + \beta = 1$. 
Then 
the following hold:
\begin{enumerate}[\rm(i)]
\item The case $q \equiv 0 \pmod 2$ (Characteristic $p=2$): If $\sqrt{q} \equiv 2 \pmod 3$, then the solutions are:
$$\begin{cases} \alpha = \omega^{\frac{q-1}{3}} \\ \beta = \alpha^{-1} = \omega^{\frac{2(q-1)}{3}} \end{cases} \quad \text{or} \quad \begin{cases} \alpha = \omega^{\frac{2(q-1)}{3}} \\ \beta = \alpha^{-1} = \omega^{\frac{q-1}{3}} \end{cases}$$
If $\sqrt{q} \equiv 1 \pmod 3$, then there is no solution.
\item The case $q \equiv 0 \pmod 3$ (Characteristic $p=3$): The only solution is $\alpha = \beta = 2$.
\item The other cases ($p \neq 2, 3$): If $\sqrt{q} \equiv 2 \pmod 3$, then the solutions are:
$$\begin{cases} \alpha = \omega^{\frac{q-1}{6}} \\ \beta = \alpha^{-1} = \omega^{\frac{5(q-1)}{6}} \end{cases} \quad \text{or} \quad \begin{cases} \alpha = \omega^{\frac{5(q-1)}{6}} \\ \beta = \alpha^{-1} = \omega^{\frac{q-1}{6}} \end{cases}$$
If $\sqrt{q} \equiv 1 \pmod 3$, then there is no solution.
\end{enumerate}
\end{lemma}
\begin{proof} 
Since $\alpha+\beta=1$, we have $\alpha\beta=1$ by Lemma \ref{sum1prod1}. 

\noindent (i) By Lemma \ref{prodsum1}, $\{\alpha,\beta\}=
\{\omega^{\frac{q-1}{3}},\omega^{\frac{2(q-1)}{3}}\}\subseteq \mu_3$. By Lemma \ref{3rdroot}, (i) follows.

\noindent (ii) follows from Lemma \ref{prodsum1}.

\noindent (iii) By Lemma \ref{prodsum1}, $\{\alpha,\beta\}=
\{\omega^{\frac{q-1}{6}},\omega^{-\frac{q-1}{6}}\}$. 
Notice that $\omega^{\pm \frac{q-1}{6}} \in \mu_{\sqrt{q}+1}$ if and only if 
$\left(\omega^{\pm \frac{q-1}{6}}\right)^{\sqrt{q}+1} = 1$, 
which is equivalent to $\sqrt{q}+1 \equiv 0 \pmod{6}$. 
Since $q$ is odd, this is further equivalent to $\sqrt{q}+1 \equiv 0 \pmod{3}$. Now (iii) follows from Lemma \ref{3rdroot}.
\end{proof}
\begin{lemma}\label{conja+b=1}
Let $\alpha$ and $\beta$ be elements in $\mathbb{F}_q$ such that $\alpha^{\sqrt{q}} = \beta^{-1}$ and $\alpha + \beta = 1$. Then the following hold:
\begin{enumerate}[\rm (i)]
\item Case $q \equiv 0 \pmod 2$: If $\sqrt{q} \equiv 1 \pmod 3$, the solutions are:
$$\begin{cases} \alpha = \omega^{\frac{q-1}{3}} \\ \beta = \alpha^{-1} = \omega^{\frac{2(q-1)}{3}} \end{cases} \quad \text{or} \quad \begin{cases} \alpha = \omega^{\frac{2(q-1)}{3}} \\ \beta = \alpha^{-1} = \omega^{\frac{q-1}{3}} \end{cases}$$
If $\sqrt{q} \equiv 2 \pmod 3$, there is no solution.

\item Case $q \equiv 0 \pmod 3$: The unique solution is $\alpha = \beta = 2$, which belongs to the group of $(\sqrt{q}+1)$-th roots of unity, $\mu_{\sqrt{q}+1}$.

\item Case $p \neq 2, 3$: If $\sqrt{q} \equiv 1 \pmod 3$, the solutions are:
$$\begin{cases} \alpha = \omega^{\frac{q-1}{6}} \\ \beta = \alpha^{-1} = \omega^{\frac{5(q-1)}{6}} \end{cases} \quad \text{or} \quad \begin{cases} \alpha = \omega^{\frac{5(q-1)}{6}} \\ \beta = \alpha^{-1} = \omega^{\frac{q-1}{6}} \end{cases}$$
If $\sqrt{q} \equiv 2 \pmod 3$, there is no solution.
\end{enumerate}
\end{lemma}
\begin{proof} 
By Lemma \ref{sum1prod1}, $\alpha\beta=1$ and $\alpha,\beta \in\mathbb{F}_{\sqrt{q}}^*$. Now, using Lemma \ref{prodsum1} and Lemma \ref{3rdroot}, we can proceed with the proof in a similar manner to that of Lemma \ref{mua+b=1}.
\end{proof}
\begin{lemma}\label{ma2+ab+b2=0}
Consider the equation $\alpha^2 + \alpha\beta + \beta^2 = 0$ for $\alpha, \beta \in \mu_{\sqrt{q}+1}$.
\begin{enumerate}[\rm(i)]
\item If $\sqrt{q} \equiv 0 \pmod 3$, the equation is satisfied if and only if $\alpha = \beta$.
\item If $\sqrt{q} \equiv 1 \pmod 3$, there is no such $\alpha, \beta \in \mu_{\sqrt{q}+1}$ satisfying the equation.  
\item If $\sqrt{q} \equiv 2 \pmod 3$, the solutions are given by $\beta = \alpha \omega^{\frac{q-1}{3}}$ and $\beta = \alpha \omega^{\frac{2(q-1)}{3}}$.
\end{enumerate}
\end{lemma}
\begin{proof}
Since $\alpha^3 - \beta^3 = (\alpha - \beta)(\alpha^2 + \alpha\beta + \beta^2) = 0$, we have $\beta = \alpha\zeta$ for some $\zeta \in \mu_3$.
If $p = 3$, i.e., $\sqrt{q} \equiv 0\pmod{3}$, then $\zeta = 1$, which implies $\alpha = \beta$.
When $p \neq 3$, $\zeta\in\mu_3 \setminus \{1\}$, and  we obtain (ii) and (iii), by Lemma \ref{3rdroot}.
\end{proof}
\begin{lemma}\label{conja2+ab+b2=0}
Consider the equation $\alpha^2 + \alpha\beta + \beta^2 = 0$ for $\alpha, \beta \in \mathbb{F}_q$ such that $\alpha^{\sqrt{q}} = \beta^{-1}$. Then the following hold:
\begin{enumerate}[\rm(i)]
\item If $\sqrt{q} \equiv 0 \pmod 3$, the equation is satisfied if and only if $\alpha = \beta$. In this case, $\alpha$ (and thus $\beta$) necessarily belongs to $\mu_{\sqrt{q}+1}$.
\item If $\sqrt{q} \equiv 2 \pmod 3$, there are no such $\alpha$ and $\beta$ satisfying the given conditions.
\item If $\sqrt{q} \equiv 1 \pmod 3$, the solutions are given by $\beta = \alpha \omega^{\frac{q-1}{3}}$ or $\beta = \alpha \omega^{\frac{2(q-1)}{3}}$. In this case, the number of such $(\alpha,\beta)$'s is $2(\sqrt{q}+1)$.
\end{enumerate}
\end{lemma}
\begin{proof}
As in the proof of Lemma~\ref{ma2+ab+b2=0}, we have $\beta = \alpha\zeta$ for some $\zeta \in \mu_3$.
When $p = 3$, we have $\zeta = 1$ and $\alpha = \beta$, which implies $\alpha \in \mu_{\sqrt{q}+1}$.
When $p \neq 3$, we have $\zeta \in \mu_3 \setminus \{1\}$. Substituting $\beta = \alpha\zeta$ into $\alpha^{\sqrt{q}} = \beta^{-1}$ yields $\zeta = \alpha^{ \sqrt{q}+1 } \in \mathbb{F}_{\sqrt{q}}$. Thus, $\zeta \in \mathbb{F}_{\sqrt{q}}^*$.

Now, applying Lemma~\ref{3rdroot}, we obtain (ii) and the first assertion of (iii). 
For each $\zeta \in \mu_3 \setminus \{1\} = \{\omega^{\pm \frac{q-1}{3}}\}$, the equation $\alpha^{\sqrt{q}+1} = \zeta$ in $\alpha$ has $\sqrt{q}+1$ solutions in $\mathbb{F}_q$. 
Hence, the second assertion of (iii) follows.
\end{proof}
\begin{lemma}\label{S1S2S3}
Let three sets as follows;
\begin{equation}
\begin{split}
S_1 &= \{(\alpha, \beta, \gamma) \in (\mu_{\sqrt{q}+1})^3 \mid \alpha, \beta, \gamma \text{ are roots of } f_c \text{ for some } c \in \mathbb{F}q^* \}  \\
S_2 &= \{(\alpha, \beta, \gamma) \in (\mu_{\sqrt{q}+1})^3 \mid \alpha+\beta \neq 1, \ \alpha^2 +\alpha\beta+\beta^2 \neq 0, \ \gamma = \frac{\alpha+\beta-\alpha\beta}{\alpha+\beta-1} \} \\
S_3 &= \{(\alpha, \beta, \gamma) \in (\mu_{\sqrt{q}+1})^3 \mid \alpha+\beta = 1, \ \gamma \neq -1 \}.
\end{split}
\end{equation}
Then $S_1$ is the disjoint union of $S_2$ and $S_3$.
\end{lemma}
\begin{proof}

Let $(\alpha, \beta, \gamma) \in S_1$. From the relations $\alpha + \beta + \gamma = c^{\sqrt{q}}$ and $\alpha\beta + \beta\gamma + \alpha\gamma = c^{\sqrt{q}}$, we obtain the identity:
$$\gamma(\alpha + \beta - 1) = \alpha + \beta - \alpha\beta$$

Case 1: $\alpha + \beta \neq 1$. In this case, $\gamma = \frac{\alpha + \beta - \alpha\beta}{\alpha + \beta - 1}$. Since $c \neq 0$, the relation $c^{\sqrt{q}} = \frac{\alpha^2 + \alpha\beta + \beta^2}{\alpha + \beta - 1}$ implies $\alpha^2 + \alpha\beta + \beta^2 \neq 0$. Thus, $(\alpha, \beta, \gamma) \in S_2$.

Case 2: $\alpha + \beta = 1$ . The relation $c^{\sqrt{q}} = \alpha + \beta + \gamma = 1 + \gamma$ implies $\gamma \neq -1$ since $c \neq 0$. Thus, $(\alpha, \beta, \gamma) \in S_3$.

Therefore, $S_1 \subseteq S_2 \cup S_3$.

Now we prove $S_2 \cup S_3 \subseteq S_1$. If $(\alpha, \beta, \gamma) \in S_2$, let $c = \left( \frac{\alpha^2 + \alpha\beta + \beta^2}{\alpha + \beta - 1} \right)^{\sqrt{q}}$. Then $\alpha + \beta + \gamma = c^{\sqrt{q}} = \alpha\beta + \beta\gamma + \alpha\gamma$. Since $\frac{c^{\sqrt{q}}}{\alpha\beta\gamma} = \frac{1}{\alpha} + \frac{1}{\beta} + \frac{1}{\gamma} = (\alpha + \beta + \gamma)^{\sqrt{q}} = c$, we obtain $\alpha\beta\gamma = c^{\sqrt{q}-1}$. Thus, $S_2 \subseteq S_1$. If $(\alpha, \beta, \gamma) \in S_3$, then, since $\alpha + \beta = 1$ and $\alpha, \beta \in \mu_{\sqrt{q}+1}$, we have $1 = (\alpha + \beta)^{\sqrt{q}} = \frac{1}{\alpha} + \frac{1}{\beta} = \frac{1}{\alpha\beta}$, which implies $\alpha\beta = 1$. Setting $c = (1 + \gamma)^{\sqrt{q}}$ leads to $c^{\sqrt{q}} = 1 + \gamma = \alpha\beta + \beta\gamma + \alpha\gamma$ and $\alpha\beta\gamma = c^{\sqrt{q}-1}$ by the same logic as above. Thus, $S_3 \subseteq S_1$. Consequently, $S_2 \cup S_3 \subseteq S_1$.

Thus the equality $S_1 = S_2 \cup S_3$ is proved.
\end{proof}
\begin{lemma}\label{polysq03}
Assume that $\sqrt{q} \equiv 0 \pmod{3}$. Under this condition, we observe the following properties regarding the  polynomials $f_c$:

\begin{enumerate}[\rm(i)]
    \item There exists no polynomial $f_c$ with $c \neq 0$ that possesses a triple root.
    
    \item The number of polynomials $f_c$ having a double root is $\sqrt{q}$. In this case, both the double root and the remaining simple root belong to
    $\mu_{\sqrt{q}+1}$.
    
    \item The number of polynomials $f_c$ that have three distinct roots, all of which belong to
    $\mu_{\sqrt{q}+1}$,
    is given by $\dfrac{q-\sqrt{q}}{6}$.
\end{enumerate}
\end{lemma}
\begin{proof}
(i) By Theorem \ref{5polys}, there exists no $f_c$ with $c\ne 0$ that possesses a triple root. 
(ii) By Lemma \ref{p3doubleroot}, for each $\beta \in \mu_{\sqrt{q}+1} \setminus \{2\}$, if $c = (1+\beta)^{\sqrt{q}}$, then the polynomial $f_c$ possesses a double root at $2$ and a simple root at $\beta$. Furthermore, these are the only polynomials $f_c$ (with $c \neq 0$) that possess a double root in characteristic $3$.
Therefore, there are exactly $\sqrt{q}$ polynomials $f_c$ (with $c \in \mathbb{F}_q^*$) that possess a double root but not a triple root. 

(iii) For each $\beta \in \mu_{\sqrt{q}+1} \setminus \{2\}$, the triples $(2, 2, \beta)$, $(2, \beta, 2)$, and $(\beta, 2, 2)$ are elements of $S_1$, defined in Lemma \ref{S1S2S3}, and correspond to the roots of the same polynomial $f_c$. Consequently, the total number of triples in $S_1$ arising from these $\sqrt{q}$ polynomials with double roots is $3\sqrt{q}$.
Since $S_1 = S_2 \cup S_3$ is a disjoint union, we determine $|S_1|$ by calculating $|S_2|$ and $|S_3|$ independently.

1. Cardinality of $S_2$: The elements of $S_2$ are determined by pairs $(\alpha, \beta) \in \mu_{\sqrt{q}+1}^2$ satisfying $\alpha + \beta - 1 \neq 0$ and $\alpha^2 + \alpha\beta + \beta^2 \neq 0$. By Lemma \ref{mua+b=1}, if $\alpha + \beta = 1$, then $\alpha\beta = 1$, which implies $\alpha^2 + \alpha\beta + \beta^2 = (\alpha+\beta)^2 - \alpha\beta = 1 - 1 = 0$. Thus, it suffices to count pairs satisfying $\alpha^2 + \alpha\beta + \beta^2 \neq 0$. According to Lemma \ref{ma2+ab+b2=0}, the condition $\alpha^2 + \alpha\beta + \beta^2 = 0$ holds if and only if $\alpha = \beta$ (in characteristic 3). Since there are $\sqrt{q}+1$ such pairs, we have:
$$|S_2| = (\sqrt{q}+1)^2 - (\sqrt{q}+1) = q + \sqrt{q}$$

2. Cardinality of $S_3$: For $S_3$, Lemma \ref{mua+b=1} states that $\alpha + \beta = 1$ implies $\alpha = \beta = 2$. For this fixed pair, there are exactly $\sqrt{q}$ choices for $\gamma \in \mu_{\sqrt{q}+1} \setminus \{-1\}$. Thus: 
$$|S_3| = \sqrt{q}.$$

3. Total Cardinality and Number of Polynomials: Therefore, the total number of triples in $S_1$ is:
$$|S_1| = (q + \sqrt{q}) + \sqrt{q} = q + 2\sqrt{q}.$$
As previously noted, $S_1$ contains no triples with three identical components. The number of triples with exactly two identical components is $3\sqrt{q}$. Subtracting these, the number of triples consisting of three distinct elements is:$$(q + 2\sqrt{q}) - 3\sqrt{q} = q - \sqrt{q}$$Since each such triple has $3! = 6$ permutations that correspond to the same set of roots for a single polynomial $f_c$, the number of distinct polynomials $f_c$ possessing three distinct roots in $\mu_{\sqrt{q}+1}$ is:
$$\frac{q - \sqrt{q}}{6}.$$
\end{proof}
\begin{lemma}\label{polysq13}
Assume that $\sqrt{q} \equiv 1 \pmod{3}$. Under this condition, we observe the following properties regarding the characteristic polynomials $f_c$:

\begin{enumerate}[\rm(i)]
    \item There exists a unique polynomial $f_c$ with $c \neq 0$ that has a triple root. In this case, the triple root is 1 and $c=3$ (or $c=1$ if $p=2$).
    
    \item The number of polynomials $f_c$ having a double root is $\sqrt{q}$. In this case, both the double root and the remaining simple root belong to
    $\mu_{\sqrt{q}+1}$.
    
    \item The number of polynomials $f_c$ that have three distinct roots, all of which belong to
    $\mu_{\sqrt{q}+1}$,
    is given by $\dfrac{q-\sqrt{q}}{6}$.
\end{enumerate}
\end{lemma}
\begin{proof}
The proof follows the same lines as that of Lemma \ref{polysq03} with slight modifications. (i) By Theorem \ref{5polys}, there exists a unique $f_c$ with $c \neq 0$ that possesses a triple root. Specifically, for $c=3$, the triple root is $\alpha=1$, and $(1,1,1) \in S_1$. 

(ii) By Lemma \ref{doublealpha} and Lemma \ref{p3doubleroot}, there are exactly $\sqrt{q}$ polynomials $f_c$ (with $c \neq 0$), each of which possesses a double root $\alpha \in \mu_{\sqrt{q}+1} \setminus \{1\}$ and a third root $\frac{\alpha(2-\alpha)}{2\alpha - 1} \in \mu_{\sqrt{q}+1}$. 

(iii) The total number of triples in $S_1$ arising from these $\sqrt{q}$ polynomials with double roots is $3\sqrt{q}$.
Since $S_1 = S_2 \cup S_3$ is a disjoint union, we determine $|S_1|$ by calculating $|S_2|$ and $|S_3|$ independently.

1. Cardinality of $S_2$: The elements of $S_2$ are determined by pairs $(\alpha, \beta) \in \mu_{\sqrt{q}+1}^2$ satisfying $\alpha + \beta - 1 \neq 0$ and $\alpha^2 + \alpha\beta + \beta^2 \neq 0$. By Lemma \ref{mua+b=1}, if $\alpha + \beta = 1$, then $\alpha\beta = 1$, which implies $\alpha^2 + \alpha\beta + \beta^2 = 0$. Thus, it suffices to count pairs satisfying $\alpha^2 + \alpha\beta + \beta^2 \neq 0$. According to Lemma \ref{ma2+ab+b2=0}, the condition $\alpha^2 + \alpha\beta + \beta^2 = 0$ has no solutions. Thus  we have:
$$|S_2| = (\sqrt{q}+1)^2 $$

2. Cardinality of $S_3$: For $S_3$, Lemma \ref{mua+b=1} states that $\alpha + \beta = 1$ has no solution, and hence 
$$|S_3| =0$$

3. Total Cardinality and Number of Polynomials: Therefore, the total number of triples in $S_1$ is:
$$|S_1| = ( \sqrt{q}+1)^2$$
As previously noted, $S_1$ contains only one triple $(1,1,1)$ with three identical components. The number of triples with exactly two identical components is $3\sqrt{q}$. Subtracting these, the number of triples consisting of three distinct elements is:
$$( 2\sqrt{q}+1)^2 -1 - 3\sqrt{q} = q - \sqrt{q}.$$
Thus the number of distinct polynomials $f_c$ possessing three distinct roots in $\mu_{\sqrt{q}+1}$ is:
$$\frac{q - \sqrt{q}}{6}.$$
\end{proof}
\begin{lemma}\label{polysq23}
Assume that $\sqrt{q} \equiv 2 \pmod{3}$. Under this condition, we observe the following properties regarding the  polynomials $f_c$:

\begin{enumerate}[\rm(i)]
    \item There exists a unique  polynomial $f_c$ with $c \neq 0$ that has a triple root. In this case, the triple root is 1 and $c=3$ (or $c=1$ if $p=2$).  
    
    \item The number of polynomials $f_c$ having a double root is $\sqrt{q}$. In this case, both the double root and the remaining simple root belong to
    $\mu_{\sqrt{q}+1}$.
    
    \item The number of polynomials $f_c$ that have three distinct roots, all of which belong to
    $\mu_{\sqrt{q}+1}$,
    is given by $\dfrac{q-\sqrt{q}-2}{6}$.
\end{enumerate}
\end{lemma}
\begin{proof}
The proof follows the same lines as that of Lemma \ref{polysq03} and \ref{polysq13} with slight modifications. (i) By Theorem \ref{5polys}, there exists a unique $f_c$ with $c \neq 0$ that possesses a triple root. Specifically, for $c=3$, the triple root is $\alpha=1$, and $(1,1,1) \in S_1$. 

(ii) By Lemma \ref{doublealpha} and Lemma \ref{p3doubleroot}, there are exactly $\sqrt{q}$ polynomials $f_c$ (with $c \neq 0$), each of which possesses a double root $\alpha \in \mu_{\sqrt{q}+1} \setminus \{1\}$ and a third root $\frac{\alpha(2-\alpha)}{2\alpha - 1} \in \mu_{\sqrt{q}+1}$. 

(iii) The total number of triples in $S_1$ arising from these $\sqrt{q}$ polynomials with double roots is $3\sqrt{q}$.
Since $S_1 = S_2 \cup S_3$ is a disjoint union, we determine $|S_1|$ by calculating $|S_2|$ and $|S_3|$ independently.

1. Cardinality of $S_2$: The elements of $S_2$ are determined by pairs $(\alpha, \beta) \in \mu_{\sqrt{q}+1}^2$ satisfying $\alpha + \beta - 1 \neq 0$ and $\alpha^2 + \alpha\beta + \beta^2 \neq 0$. By Lemma \ref{mua+b=1}, if $\alpha + \beta = 1$, then $\alpha\beta = 1$, which implies $\alpha^2 + \alpha\beta + \beta^2 = 0$. Thus, it suffices to count pairs satisfying $\alpha^2 + \alpha\beta + \beta^2 \neq 0$. According to Lemma \ref{ma2+ab+b2=0}, the condition $\alpha^2 + \alpha\beta + \beta^2 = 0$ has $(\sqrt{q}+1)$ solutions $(\alpha,\beta$ with $\beta=\alpha \omega^{\frac{q-1}{3}}$ or $\beta=\alpha \omega^{\frac{-(q-1)}{3}}$ and $\alpha\in\mu_{\sqrt{q}+1}$. Thus  we have:
$$|S_2| = (\sqrt{q}+1)^2 -2(\sqrt{q}+1)=q-1$$

2. Cardinality of $S_3$: For $S_3$, Lemma \ref{mua+b=1} states that $\alpha + \beta = 1$ has two pairs $(\alpha,\beta)$ as solutions, and hence 
$$|S_3| =2\sqrt{q}$$

3. Total Cardinality and Number of Polynomials: Therefore, the total number of triples in $S_1$ is:
$$|S_1| = q-1+2\sqrt{q}.$$
As previously noted, $S_1$ contains only one triple $(1,1,1)$ with three identical componets. The number of triples with exactly two identical components is $3\sqrt{q}$. Subtracting these, the number of triples consisting of three distinct elements is:
$$(q-1+2\sqrt{q}) -1 - 3\sqrt{q} = q - \sqrt{q}-2.$$
Thus the number of distinct polynomials $f_c$ possessing three distinct roots in $\mu_{\sqrt{q}+1}$ is:
$$\frac{q - \sqrt{q}-2}{6}.$$
\end{proof}
\begin{lemma}\label{conj3roots}
The number of characteristic polynomials that have three distinct simple roots, where only one of them belongs to
$\mu_{\sqrt{q}+1}$,
is:
\begin{equation*}
\begin{cases}
\dfrac{q-\sqrt{q}-2}{2} \mbox{ when }\sqrt{q} \equiv 0 \pmod{3} \\
\dfrac{q-\sqrt{q}-4}{2} \mbox{ when }\sqrt{q} \equiv 1 \pmod{3} \\
\dfrac{q-\sqrt{q}-2}{2} \mbox{ when }\sqrt{q} \equiv 2 \pmod{3} 
\end{cases}
\end{equation*}
\end{lemma}
\begin{proof}
By Theorem \ref{5polys}, if the characteristic polynomial of $A$ has three distinct roots $\alpha, \beta, \gamma$ such that only $\alpha$ is a $(\sqrt{q}+1)$-th root of unity, then $\beta$ and $\gamma$ satisfy $\beta \neq \gamma$ and $\beta^{\sqrt{q}} = \gamma^{-1}$. The number of such pairs $\{\beta, \gamma\} = \{\beta, \beta^{-\sqrt{q}}\}$ is exactly
$$\frac{q - 1 - (\sqrt{q} + 1)}{2} = \frac{q - \sqrt{q} - 2}{2}.$$
Following the argument in Lemma \ref{betgamconjfc}, for each pair $\{\beta, \gamma\}$, the value of $\alpha$ is determined based on the following conditions:If $\beta + \gamma \neq 1$ and $\beta^2 + \beta\gamma + \gamma^2 \neq 0$: The value of $\alpha$ is uniquely determined as
$$\alpha = \frac{\beta + \gamma - \beta\gamma}{\beta + \gamma - 1}.$$
If $\beta + \gamma = 1$, by Lemma \ref{betgamconjfc}, there are exactly $\sqrt{q}$ possible values for $\alpha$.

In the case where $\sqrt{q} \equiv 0 \pmod 3$, Lemma \ref{conja+b=1} and Lemma \ref{ma2+ab+b2=0} imply that $\beta + \gamma = 1$ holds if and only if $\beta = \gamma = 2$, and $\beta^2 + \beta\gamma + \gamma^2 = 0$ holds only when $\beta = \gamma$. Since we are considering the case where $\beta \neq \gamma$, 
for each pair  $(\beta, \gamma)$ of the $\frac{q - \sqrt{q} - 2}{2}$ pairs, $\alpha$ is uniquely determined. Thus, the total number of such characteristic polynomials is $\frac{q - \sqrt{q} - 2}{2}$.

In the case where $\sqrt{q} \equiv 1 \pmod{3}$, the enumeration of the characteristic polynomials is conducted as follows. First, the number of pairs $(\beta, \gamma)$ such that $\beta^{\sqrt{q}} = \gamma^{-1}$ and $\beta \neq \gamma$ is $q - \sqrt{q} - 2$.  By Remark \ref{sum1a2abb20},   the condition $\beta + \gamma = 1$ implies $\beta^2 + \beta\gamma + \gamma^2 = 0$. By Lemma \ref{conja2+ab+b2=0}, the number of pairs $(\beta, \gamma)$ satisfying both $\beta^2 + \beta\gamma + \gamma^2 \neq 0$ and $\beta + \gamma \neq 1$ is calculated by:
$$(q - \sqrt{q} - 2) - 2(\sqrt{q} + 1) = q - 3\sqrt{q} - 4.$$ 
For each such pair $(\beta, \gamma)$, the root $\alpha$ is uniquely determined.
On the other hand, by Lemma \ref{conja+b=1}, there are exactly two pairs that satisfy the condition $\beta + \gamma = 1$. For each of these pairs, Lemma \ref{betgamconjfc} implies that there are $\sqrt{q}$ possible values for $\alpha$.
Consequently, the total number of triples $(\alpha, \beta, \gamma)$ such that $f_c(\alpha) = f_c(\beta) = f_c(\gamma) = 0$ for some $c \in \mathbb{F}_q^*$ is:
$$(q - 3\sqrt{q} - 4) + 2\sqrt{q} = q - \sqrt{q} - 4.$$
Since the characteristic polynomials corresponding to the triples $(\alpha, \beta, \gamma)$ and $(\alpha, \gamma, \beta)$ are identical, the total number of distinct characteristic polynomials is
$\frac{q - \sqrt{q} - 4}{2}.$

When $\sqrt{q} \equiv 2 \pmod 3$:
By Lemma \ref{conja+b=1} and Lemma \ref{ma2+ab+b2=0}, there exist no $\beta, \gamma$ satisfying $\beta + \gamma = 1$ or $\beta^2 + \beta\gamma + \gamma^2 = 0$. Consequently, for each of the $\frac{q - \sqrt{q} - 2}{2}$ pairs, $\alpha$ is uniquely determined. Thus, the total number of such characteristic polynomials is $\frac{q - \sqrt{q} - 2}{2}$.
\end{proof}
\begin{lemma}\label{irrpoly}
The number of irreducible characteristic polynomials $f_c$, $c\in\mathbb{F}_q^*$ over $\mathbb{F}_q$  is:
\begin{equation*}
\begin{cases}
\dfrac{q-\sqrt{q}}{3} \mbox{ when }\sqrt{q} \equiv 0 \pmod{3} \\
\dfrac{q-\sqrt{q}}{3} \mbox{ when }\sqrt{q} \equiv 1 \pmod{3} \\
\dfrac{q-\sqrt{q}-2}{3} \mbox{ when }\sqrt{q} \equiv 2 \pmod{3} 
\end{cases}
\end{equation*}
\end{lemma}
\begin{proof}
By Theorem \ref{5polys}, since any characteristic polynomial must either be irreducible or possess three roots (counting multiplicities) in $\mathbb{F}_{q^s}$, the number of irreducible polynomials can be determined by subtracting the counts of all previously identified cases from the total. Consequently, their cardinality can be obtained by subtracting all other cases from the total.
\end{proof}
\begin{theorem}\label{Cctable}
 Using above lemmas, we obtain the following table:

\begin{table}[h]
\centering
\small 
\renewcommand{\arraystretch}{1.8} 
\begin{tabular}{|c|c|c|c|c|}
\hline
\multirow{2}{*}{Case} & \multirow{2}{*}{type $(n,i)$} & \multicolumn{3}{c|}{the number of polynomials} \\ \hhline{~~|---} 
 &  & $\sqrt{q} \equiv 0 \pmod{3}$ & $\sqrt{q} \equiv 1 \pmod{3}$ & $\sqrt{q} \equiv 2 \pmod{3}$ \\ \hline
(I)   & $(q-\sqrt{q}+1, 0)$ & $\frac{q-\sqrt{q}}{3}$ & $\frac{q-\sqrt{q}}{3}$ & $\frac{q-\sqrt{q}-2}{3}$ \\ \hline
(II)  & $(q+1, 1)$          & 0 & 1 & 1 \\ \hline
(III) & $(q+1, 2)$          & $\frac{q-\sqrt{q}-2}{2}$ & $\frac{q-\sqrt{q}-4}{2}$ & $\frac{q-\sqrt{q}-2}{2}$ \\ \hline
(IV)  & $(q+\sqrt{q}+1, 1)$ & $\sqrt{q}$ & $\sqrt{q}$ & $\sqrt{q}$ \\ \hline
(V)   & $(q+2\sqrt{q}+1, 0)$ & $\frac{q-\sqrt{q}}{6}$ & $\frac{q-\sqrt{q}}{6}$ & $\frac{q-\sqrt{q}-2}{6}$ \\ \hline
\end{tabular}
\end{table}
\noindent 
In the table, ``Cases'' are defined as follows;
\begin{enumerate}[label={\rm (\Roman*)}]
    \item[\rm{(I)}] $f_c$ is an irreducible polynomial.
    \item[\rm{(II)}]  $f_c$ has a triple root.
    \item[\rm{(III)}]  $f_c$ has three simple roots, only one of them is a $(\sqrt{q}+1)$-th root of unity.
    \item[\rm{(IV)}]  $f_c$ has a double root and the other simple root, both of them are $(\sqrt{q}+1)$-th roots of unity.
    \item[\rm{(V)}]  $f_c$ has three simple roots which are all $(\sqrt{q}+1)$-th roots of unity.
\end{enumerate}
\end{theorem}

\begin{theorem} For any $q$, there exist curves of exactly eight types, $(q\sqrt{q}+1,q\sqrt{q}+1)$, $(\sqrt{q}+1,\sqrt{q}+1)$, $(1,1)$, $(q-\sqrt{q}+1,0)$, $(q+1,1)$, $(q+1,2)$, $(q+\sqrt{q}+1,1)$, and $(q+2\sqrt{q}+1,0)$.
\end{theorem}
\begin{proof}
 From Theorems \ref{same}, \ref{Ca0}, and \ref{Cctable},  we obtain the result.
\end{proof}

\noindent
Acknowledgments.\quad The second author was partially supported by
the National Research Foundation of Korea(NRF) grant funded by the Korea government(MSIT) (2022R1A2C1012291).
%


\begin{thebibliography}{00}
\bibitem{bos-cha1966}
R. C. Bose and I. M. Chakravarti,
{\em Hermitian varieties in a finite projective space $PG(N, q^2)$},
Canad. J. Math. 18 (1966) 1161--1182.
\bibitem{hir1979}
J. W. P. Hirschfeld,
Projective geometries over finite fields,
Oxford University Press,
Oxford, 1979.
\bibitem{hir1998}
J. W. P. Hirschfeld,
Projective geometries over finite fields
(second edition),
Oxford University Press,
Oxford, 1998.
\bibitem{hom1987}
M. Homma,
{\em
Funny plane curves in characteristic $p>0$},
Comm. Algebra 15 (1987) 1469--1501.
\bibitem{hom2006}
M. Homma,
{\em
Galois points on a Hermitian curve},
Comm. Algebra 34 (2006) 4503–4511.

\bibitem{hom-kim2024}
M. Homma and S.J. Kim,
{\em Relatives of the Hermitian curve}, Journal of Geometry(2026) 117:11. https://doi.org/10.1007/s00022-026-00795-8

\bibitem{par1986}
R. Pardini,
{\em
Some remarks on plane curves over fields
of finite characteristic},
Compositio Math. 60 (1986) 3--17.

\end{thebibliography}
\end{document}